\documentclass[12pt]{article}
\usepackage{amsmath, amssymb, amsthm, pstricks, graphicx}

\newtheorem{prop}{Proposition}[section]
\newtheorem{cor}[prop]{Corollary}
\newtheorem{thm}[prop]{Theorem}
\newtheorem{lemma}[prop]{Lemma}
\newtheorem*{recap}{Theorem \ref{k+2thm}}

\theoremstyle{remark}
\newtheorem{NB}{Note}

\newcommand{\sh}{\operatorname{sh}}

\usepackage[bookmarks=false, naturalnames=true]{hyperref}

\title{Pattern avoidance and RSK-like algorithms for alternating permutations and Young tableaux\footnote{{This is the extended version of \cite{JCTAversion}.  In the published version, the order of Sections \ref{subsec:1stbij} and \ref{subsec:2ndbij} is reversed, and the proofs in the present Section \ref{subsec:1stbij} are omitted.  Otherwise, the papers are essentially the same.}}
}
\author{Joel Brewster Lewis \\
Massachusetts Institute of Technology \\
Cambridge, MA 02139 \\
\texttt{jblewis@math.mit.edu}}

\begin{document}

\maketitle

\begin{abstract}
We define a class $\mathcal{L}_{n, k}$ of permutations that generalizes alternating (up-down) permutations and give bijective proofs of certain pattern-avoidance results for this class.  
As a special case of our results, we give two bijections between the set $A_{2n}(1234)$ of alternating permutations of length $2n$ with no four-term increasing subsequence and standard Young tableaux of shape $\langle 3^n\rangle$, and between the set $A_{2n + 1}(1234)$ and standard Young tableaux of shape $\langle 3^{n - 1}, 2, 1\rangle$.  This represents the first enumeration of alternating permutations avoiding a pattern of length four.  We also extend previous work on doubly-alternating permutations (alternating permutations whose inverses are alternating) to our more general context.

The set $\mathcal{L}_{n, k}$ may be viewed as the set of reading words of the standard Young tableaux of a certain skew shape.  In the last section of the paper, we expand our study to consider pattern avoidance in the reading words of standard Young tableaux of any skew shape.  We show bijectively that the number of standard Young tableaux of shape $\lambda/\mu$ whose reading words avoid $213$ is a natural $\mu$-analogue of the Catalan numbers (and in particular does not depend on $\lambda$, up to a simple technical condition), and that there are similar results for the patterns $132$, $231$ and $312$.
\end{abstract}

\section{Introduction}\label{sec:intro}

A classical problem asks for the number of permutations that avoid a certain permutation pattern.  This problem has received a great deal of attention (see e.g., \cite{simion, 1342, 1234}) and has led to a number of interesting variations including the enumeration of special classes of pattern-avoiding permutations (e.g., involutions \cite{simion} and derangements \cite{der}).  One such variation, first studied by Mansour in \cite{mans}, is the enumeration of \emph{alternating} permutations avoiding a given pattern or collection of patterns.  Alternating permutations have the intriguing property \cite{mans, catadd, geehoon} that for any pattern of length three, the number of alternating permutations of a given length avoiding that pattern is given by a Catalan number.  This property is doubly interesting because it is shared by the class of all permutations.  This coincidence suggests that pattern avoidance in alternating permutations and in usual permutations may be closely related and so motivates the study of pattern avoidance in alternating permutations.

In this paper, we extend the study of pattern avoidance in alternating permutations to consider patterns of length four.  In particular, we show that the number of alternating permutations of length $2n$ avoiding the pattern $1234$ is $\frac{2 \cdot (3n)!}{n! (n + 1)! (n + 2)!}$ and that the number of alternating permutations of length $2n + 1$ avoiding $1234$ is $\frac{16 (3n)!}{(n - 1)!(n + 1)!(n + 3)!}$.  This is the first enumeration of a set of pattern-avoiding alternating permutations for a single pattern of length four.  These enumerations follow from bijections between permutations and standard Young tableaux of certain shapes.   

Our bijections work in a more general setting in which we replace alternating permutations with the set $\mathcal{L}_{n, k}$ of reading words of standard Young tableau of certain nice skew shapes.  
We also give proofs of two results that suggest that $\mathcal{L}_{n, k}$ is a ``good'' generalization of $S_n$ and $A_n$ for purposes of pattern avoidance.  The first is a natural extension to $\mathcal{L}_{n, k}$ of the work of Ouchterlony \cite{ouch} on pattern avoidance for alternating permutations whose inverses are also alternating.  The result uses RSK to relate these permutations to pattern-avoiding permutations in $S_n$, as well as to relate pattern-avoiding involutions in $\mathcal{L}_{n, k}$ to pattern-avoiding involutions in $S_n$.  The second result is an enumeration of standard skew Young tableaux of \emph{any} fixed shape whose reading words avoid certain patterns.  In particular, it provides a uniform argument to enumerate permutations in $S_n$ and permutations in $\mathcal{L}_{n, k}$ that avoid either $132$ or $213$.  That such a bijection should exist is far from obvious, and it raises the possibility that there is substantially more to be said in this area.  In the remainder of this introduction, we provide a more detailed summary of results.

Both the set of all permutations of a given length and the set of alternating permutations of a given length can be expressed as the set of reading words of the standard Young tableaux of a particular skew shape (essentially a difference of two staircases).  We define a class $\mathcal{L}_{n, k} \subseteq S_{nk}$ of permutations such that $\mathcal{L}_{n,1} = S_n$ is the set of all permutations of length $n$, $\mathcal{L}_{n,2}$ is the set of alternating permutations of length $2n$, and for each $k$ $\mathcal{L}_{n, k}$ is the set of reading words of the standard Young tableaux of a certain skew shape.  In Section \ref{defs}, we provide definitions and notation that we use throughout this paper.  In Sections \ref{k+1sec}, \ref{sec:k+2}, \ref{sec:ouch} and \ref{sec:oddlength}, we use bijective proofs to derive enumerative pattern avoidance results for $\mathcal{L}_{n,k}$.  In Section \ref{k+1sec} we give a simple bijection between elements of $\mathcal{L}_{n, k}$ with no $(k + 1)$-term increasing subsequence and standard Young tableaux of rectangular shape $\langle k^n\rangle$.  In Section \ref{sec:k+2} we exhibit two bijections between the elements of $\mathcal{L}_{n, k}$ with no $(k + 2)$-term increasing subsequence and standard Young tableaux of rectangular shape $\langle(k + 1)^n \rangle$, one of which is a modified version of the famous RSK bijection and the other of which is related to generating trees.  In Section \ref{sec:ouch}, we extend the results of \cite{ouch} to give bijections between the set of permutations of length $n$ with no $(k + 2)$-term increasing subsequence and the set of permutations $w$ with no $(k + 2)$-term increasing subsequence such that both $w$ and $w^{-1}$ lie in $\mathcal{L}_{n, k}$, as well as bijections between the involutions in each set.  In Section \ref{sec:oddlength}, we define a more general class $\mathcal{L}_{n, k;r}$ of permutations with the property that $\mathcal{L}_{n, k; 0} = \mathcal{L}_{n, k}$ and $\mathcal{L}_{n, 2; 1} = A'_{2n + 1}$, the set of down-up alternating permutations of length $2n + 1$.  We describe the appropriate changes to the bijections from the earlier sections that allow them to apply in this setting, although we do not provide proofs of their correctness in this case.

In Section \ref{skewsec}, we broaden our study to arbitrary skew shapes and so initiate the study of pattern avoidance in reading words of skew tableaux of any shape.  In Section \ref{213sec}, we show bijectively that the number of tableaux of shape $\lambda / \mu$ (under a technical restriction on the possible shapes that sacrifices no generality -- see Note \ref{NB}) whose reading words avoid $213$ can be easily computed from the shape.  Notably, the resulting value does not depend on $\lambda$ and is in fact a natural $\mu$-generalization of the Catalan numbers.  Replacing $213$ with $132$, $231$ or $312$ leads to similar results.  In Section \ref{123sec}, we show that enumerating tableaux of a given skew shape whose reading words avoid $123$ can be reduced to enumerating permutations with a given descent set that avoid $123$.  The same result holds for the pattern $321$, but we omit the proof.

\section{Definitions}\label{defs}

A \emph{permutation} $w$ of length $n$ is a word containing each of the elements of $[n]:=\{1, 2, \ldots, n\}$ exactly once.  The set of permutations of length $n$ is denoted $S_n$.  Given a word $w = w_1w_2 \cdots w_n$ and a permutation $p = p_1\cdots p_k \in S_k$, we say that $w$ \emph{contains the pattern $p$} if there exists a set of indices $1 \leq i_1 < i_2 < \ldots < i_k \leq n$ such that the subsequence $w_{i_1}w_{i_2}\cdots w_{i_k}$ of $w$ is order-isomorphic to $p$, i.e., $w_{i_\ell} < w_{i_m}$ if and only if $p_\ell < p_m$.  Otherwise, $w$ is said to \emph{avoid} $p$.  Given a pattern $p$ and a set $S$ of permutations, we denote by $S(p)$ the set of elements of $S$ that avoid $p$.  For example, $S_n(123)$ is the set of permutations of length $n$ avoiding the pattern $123$, i.e., the set of permutations with no three-term increasing subsequence.

A permutation $w = w_1w_2\cdots w_n$ is \emph{alternating} if $w_{1} < w_{2} > w_3 < w_4 > \ldots$.  (Note that in the terminology of \cite{EC1}, these ``up-down'' permutations are \emph{reverse alternating} while alternating permutations are ``down-up'' permutations.  Luckily, this convention doesn't matter: any pattern result on either set can be translated into a result on the other via \emph{complementation}, i.e., by considering $w^c$ such that $w^c_i = n + 1 - w_i$.  Then results for the pattern $132$ and up-down permutations would be replaced by results $312$ and down-up permutations, and so on.)  We denote by $A_n$ the set of alternating permutations of length $n$.

For integers $n, k \geq 1$, define $\mathcal{L}_{n, k} \subseteq S_{nk}$ to be the set of permutations $w = w_{1,1}w_{1,2}\cdots w_{1,k}w_{2,1}\cdots w_{n,k}$ that satisfy the conditions
\begin{description}
\item[\hypertarget{L1}{L1.}] $w_{i,j} < w_{i,j + 1}$ for all $1 \leq i \leq n$, $1 \leq j \leq k - 1$, and
\item[\hypertarget{L2}{L2.}] $w_{i,j + 1} > w_{i+1, j}$ for all $1 \leq i \leq n - 1$, $1 \leq j \leq k - 1$.
\end{description}
Note in particular that $\mathcal{L}_{n, 1} = S_n$ (we have no restrictions in this case) and $\mathcal{L}_{n, 2} = A_{2n}$.  For any $k$ and $n$, $|\mathcal{L}_{n, k}(12\cdots k)| = 0$.  Thus, for monotone pattern-avoidance in $\mathcal{L}_{n,k}$ we should consider patterns of length $k + 1$ or longer.  The set $\mathcal{L}_{n, k}$ has been enumerated by Baryshnikov and Romik \cite{Romik}, and the formulas that result are quite simple for small values of $k$.

\begin{NB}\label{asdescents}
If $w = w_{1,1}\cdots w_{n, k} \in S_{nk}$ satisfies \hyperlink{L1}{L1} and also avoids $12\cdots (k + 2)$ then it automatically satisfies \hyperlink{L2}{L2}, since a violation $w_{i, j + 1} < w_{i + 1, j}$ of \hyperlink{L2}{L2} leads immediately to a $(k + 2)$-term increasing subsequence $w_{i, 1} < \ldots < w_{i, j + 1} < w_{i + 1, j} < \ldots < w_{i + 1, k}$.  Consequently, we can also describe $\mathcal{L}_{n, k}(1\cdots (k + 2))$ (respectively, $\mathcal{L}_{n, k}(1 \cdots (k + 1))$) as the set of permutations in $S_{nk}(1\cdots (k + 2))$ (respectively, $S_{nk}(1\cdots (k + 1))$) whose descent set is (or in fact, is contained in) $\{k, 2k, \ldots, (n - 1)k\}$.  Thus, in Sections \ref{k+1sec}, \ref{sec:k+2} and \ref{sec:ouch} we could replace $\mathcal{L}_{n, k}$ by permutations with descent set $\{k, 2k, \ldots\}$ without changing the content of any theorems.
\end{NB}

A \emph{partition} is a weakly decreasing, finite sequence of nonnegative integers.  We consider two partitions that differ only in the number of trailing zeroes to be the same.  We write partitions in sequence notation, as $\langle \lambda_1, \lambda_2, \ldots, \lambda_n \rangle$, or to save space, with exponential notation instead of repetition of equal elements.  Thus, the partition $\langle 5, 5, 3, 3, 2, 1\rangle$ may be abbreviated $\langle 5^2, 3^2, 2, 1\rangle$.  If the sum of the entries of $\lambda$ is equal to $m$ then we write $\lambda \vdash m$.

Given a partition $\lambda = \langle \lambda_1, \lambda_2, \ldots, \lambda_n \rangle$, the \emph{Young diagram} of shape $\lambda$ is a left-justified array of $\lambda_1 + \ldots + \lambda_n$ boxes with $\lambda_1$ in the first row, $\lambda_2$ in the second row, and so on.  We identify each partition with its Young diagram and speak of them interchangeably.  A \emph{skew Young diagram} $\lambda / \mu$ is the diagram that results when we remove the boxes of $\mu$ from those of $\lambda$, when the diagrams are arranged so that their first rows and first columns coincide.  If $\lambda / \mu$ is a skew Young diagram with $m$ boxes, a \emph{standard Young tableau} of shape $\lambda / \mu$ is a filling of the boxes of $\lambda / \mu$ with $[m]$ so that each element appears in exactly one box, and entries increase along rows and columns.  We denote by $\sh(T)$ the shape of the standard Young tableau $T$ and by $f^{\lambda/\mu}$ the number of standard Young tableaux of shape $\lambda/\mu$.  We identify boxes in a (skew) Young diagram using matrix coordinates, so the box in the first row and second column is numbered $(1, 2)$.  

Given a standard Young tableau, we can form the \emph{reading word} of the tableau by reading the last row from left to right, then the next-to-last row, and so on.  In English notation (with the first row at the top and the last row at the bottom), this means we read rows left to right, starting with the bottom row and working up.  For example, the reading words of the tableaux of shape $\langle n, n - 1, \ldots, 2, 1\rangle / \langle n - 1, n - 2, \ldots, 1\rangle$ are all of $S_n$, and similarly $\mathcal{L}_{n, k}$ is equal to the set of reading words of standard skew Young tableaux of shape $\langle n + k - 1, n + k - 2, \ldots, k\rangle / \langle n - 1, n - 2, \ldots, 1\rangle$, as illustrated in Figure \ref{readingword}.  The other ``usual" reading order, from right to left then top to bottom in English notation, is simply the reverse of our reading order.   Consequently, any pattern-avoidance result in our case carries over to the other reading order by taking the \emph{reverse} of all permutations and patterns involved, i.e., by replacing $w = w_1\cdots w_n$ with $w^r = w_n \cdots w_1$.

\begin{figure}
\begin{center}
\input{tableau.tex}
\end{center}
\caption{A standard skew Young tableau whose reading word is the permutation $7 \, 10 \, 14 \, 8 \, 13 \, 15 \, 4 \, 11 \, 12 \, 1 \, 5 \, 9 \, 2 \, 3 \, 6 \in \mathcal{L}_{5,3}$.}
\label{readingword}
\end{figure}

We make note of two more operations on Young diagrams and tableaux.  Given a partition $\lambda$, the \emph{conjugate partition} $\lambda'$ is defined so that the $i$th row of $\lambda'$ has the same length as the $i$th column of $\lambda$ for all $i$.  The conjugate of a skew Young diagram $\lambda / \mu$ is defined by $(\lambda / \mu)' = \lambda' / \mu'$.  Given a standard skew Young tableau $T$ of shape $\lambda / \mu$, the conjugate tableau $T'$ of shape $(\lambda / \mu)'$ is defined to have the entry $a$ in box $(i, j)$ if and only if $T$ has the entry $a$ in box $(j, i)$.  Geometrically, all these operations can be described as ``reflection through the main diagonal.''  Given a skew Young diagram $\lambda / \mu$, rotation by $180^\circ$ gives a new diagram $(\lambda / \mu)^*$.  Given a tableaux $T$ with $n$ boxes, we can form $T^*$, the \emph{rotated-complement} of $T$, by rotating $T$ by $180^\circ$ and replacing the entry $i$ with $n + 1 - i$ for each $i$.  Observe that the reading word of $T^*$ is exactly the reverse-complement of the reading word of $T$.

The \emph{Schensted insertion algorithm}, or equivalently the \emph{RSK correspondence}, is an extremely powerful tool relating permutations to pairs of standard Young tableaux.  For a description of the bijection and a proof of its correctness and some of its properties, we refer the reader to \cite[Chapter 7]{EC2}.  Our use of notation follows that source, so in particular we denote by $T \leftarrow i$ the tableau that results when we (row-) insert $i$ into the tableau $T$.  Particular properties of RSK will be quoted as needed.

\section{The pattern $12\cdots (k + 1)$}\label{k+1sec}

In this section we give the simplest of the bijections in this paper.

\begin{prop}\label{k+1}
 There is a bijection between $\mathcal{L}_{n, k}(12\cdots (k + 1))$ and the set of standard Young tableaux of shape $\langle k^n\rangle$. 
\end{prop}
 
We have $f^{\langle n \rangle} = f^{\langle 1^n \rangle} = 1$ and $f^{\langle n, n \rangle} = f^{\langle 2^n \rangle} = \frac{1}{n + 1}\binom{2n}{n} = C_n$, the $n$th Catalan number.  By the hook-length formula \cite{EC2, sagan} we have 
\[
f^{\langle k^n \rangle} = \frac{(kn)! \cdot 1! \cdot 2! \cdots (k - 1)!}{n! \cdot (n + 1)! \cdots (n + k - 1)!}.
\]
Then Proposition \ref{k+1} says $|\mathcal{L}_{n, k}(12\cdots (k + 1))| = f^{\langle k^n \rangle}$.  For $k = 1$, this is the uninspiring result $|S_n(12)| = 1$.  For $k = 2$, it tells us $|A_{2n}(123)| = C_n$, a result that Stanley \cite{catadd} attributes to Deutsch and Reifegerste.
\begin{proof} 
Given a permutation $w \in \mathcal{L}_{n, k}(12\cdots (k + 1))$, define a tableau $T$ of shape $\langle k^n \rangle$ by $T_{i,j} = w_{n+1-i,j}$.  We claim that this is the desired bijection.  By \hyperlink{L1}{L1} we have $T_{i, j} = w_{n- i + 1,j} < w_{n - i + 1,j+1} = T_{i, j + 1}$ for all $1 \leq i \leq n$, $1 \leq j \leq k - 1$.  Thus, the tableau $T$ is increasing along rows.  Suppose for sake of contradiction that there is some place at which $T$ fails to increase along a column.  Then there are some $i, j$ such that $w_{n - i + 1, j} = T_{i, j} > T_{i + 1, j} = w_{n - i, j}$.   Then $w_{n - i, 1} < w_{n - i, 2} < \ldots < w_{n - i, j} < w_{n - i + 1, j} < w_{n - i + 1, j + 1} < \ldots < w_{n - i + 1, k}$ is an instance of $12\cdots (k + 1)$ contained in $w$, a contradiction, so $T$ must also be increasing along columns.  Since $w$ contains each of the values between $1$ and $nk$ exactly once, $T$ does as well, so $T$ is a standard Young tableau.  

Conversely, suppose we have a standard Young tableau $T$ of shape $\langle k^n\rangle$.  Define a permutation $w = w_{1,1}\cdots w_{n,k} \in S_{nk}$ by $w_{i,j} = T_{n - i + 1, j}$.  If we show that the image of this map lies in $\mathcal{L}_{n,k}(1 \cdots (k +1))$ we will be done, as in this case the two maps are clearly inverses.  Note that
\begin{equation}\label{cols}
w_{i + 1, j} = T_{n - i, j} < T_{n -i + 1, j} = w_{i, j}
\end{equation}
and
\begin{equation}\label{rows}
w_{i, j} = T_{n - i + 1, j} < T_{n - i + 1, j + 1} = w_{i, j + 1}
\end{equation}
because $T$ is increasing along columns and along rows, respectively.  Condition \ref{rows} is exactly \hyperlink{L1}{L1}, while the two conditions together give us that $w_{i + 1, j} < w_{i, j} < w_{i, j + 1}$, and this is \hyperlink{L2}{L2}.  Now choose any $k + 1$ entries of $w$.  By the pigeonhole principle, this set must include two entries with the same second index, and by Equation \ref{cols} these two will form an inversion.  Thus no subsequence of $w$ of length $k + 1$ is monotonically increasing and so $w$ avoids the pattern $12\cdots (k +1)$.  Thus $w \in \mathcal{L}_{n,k}(1\cdots (k + 1))$ and we have a bijection between the set $\mathcal{L}_{n, k}(12\cdots (k + 1))$ and the set of Young tableaux of shape $\langle k^n\rangle$, as desired.
\end{proof}

For example, we have under this bijection the correspondences
\begin{center}
\begin{tabular}{| c | c | c |}
\hline
1 & 3 & 5 \\
\hline
2 & 6 & 7 \\
\hline
4 & 8 & 9 \\
\hline
\end{tabular}
\hspace{.5cm}
$\longleftrightarrow$ 
\hspace{.5cm} 
$489267135 \in \mathcal{L}_{3,3}(1234)$
\end{center}

and 

\begin{center}
\begin{tabular}{|c|c|}
\hline
1 & 3 \\
\hline
2 & 5 \\
\hline
4 & 7 \\

\hline
6 & 8\\
\hline
\end{tabular}
\hspace{.5cm}
$\longleftrightarrow$ 
\hspace{.5cm} 
$68472513 \in \mathcal{L}_{4,2}(123)$.
\end{center}

Both directions of this bijection are more commonly seen with other names.  The map that sends $w \mapsto T$ is actually the Schensted insertion algorithm used in the RSK correspondence.  (For any $w \in \mathcal{L}_{n, k}(1\cdots (k + 1))$, the recording tableau is the tableau whose first row contains $\{1, \ldots, k\}$, second row contains $\{k + 1, \ldots, 2k\}$, and so on.)  The map that sends $T \mapsto w$ is the reading-word map defined in Section \ref{defs}.

\section{The pattern $12\cdots (k +1)(k + 2)$}\label{sec:k+2}

There are several nice proofs of the equality $|S_n(123)| = C_n$ including a clever application of the RSK algorithm (\cite[Problem 6.19(ee)]{EC2}).  In this section, we give two bijective proofs of the following generalization of this result.

\begin{thm}\label{k+2thm} There is a bijection between $\mathcal{L}_{n, k}(12\cdots (k + 1)(k + 2))$ and the set of standard Young tableaux of shape $\langle(k + 1)^n\rangle$ and so 
\[
|\mathcal{L}_{n, k}(12\cdots (k + 1)(k + 2))| = f^{\langle (k + 1)^n \rangle}.
\]
\end{thm}  

For $k = 1$ this is a rederivation of the equality $|S_n(123)| = C_n$ while for $k = 2$ it implies the following.

\begin{cor} We have
$\displaystyle |A_{2n}(1234)| = f^{\langle 3^n \rangle} = \frac{2(3n)!}{n!(n + 1)!(n + 2)!}$ for all $n \geq 0$.
\end{cor}

We believe this to be the first computation of any expression of the form $A_{2n}(\pi)$ or $A_{2n + 1}(\pi)$ for $\pi \in S_4$.  The complementary result giving the value of $A_{2n + 1}(1234)$ can be found in Section \ref{sec:oddlength}.

The first of the two bijections is a novel recursive bijection that may be thought of as an isomorphism between generating trees for permutations and tableaux.\footnote{In particular, the intermediate object associated to a permutation $w$ by our bijection is essentially the same as the sequence of labels in the generating tree for $\mathcal{L}_{n, k}(1\cdots(k + 2))$ connecting the root to $w$.  This generating tree is closely related to the generating tree for $S_n(1234)$ discussed in \cite{west1995} and \cite{kernelgentrees}.}  
As such, the proof of its correctness is somewhat technical and not entirely enlightening.  The second bijection makes use of Schensted insertion and so is more transparent: many of the technical results that would otherwise be necessary are well-known or are part of the standard proofs of correctness of RSK.

\subsection{A first bijection}\label{subsec:1stbij}

We give a two-step bijection, using as an intermediate set a certain family of tableaux that we will call \emph{good tableaux}.  First, we define good tableaux, construct maps between them and $\mathcal{L}_{n, k}(1\cdots (k + 2))$ and show that these maps are inverse bijections.

A tableau $T$ of shape $\langle(k + 1)^n \rangle$ (i.e., a set $\{T_{i, j}\}$ of positive integers for $1 \leq i \leq n$, $1 \leq j \leq k + 1$) is said to be \emph{good} if it satisfies the following conditions:
\begin{description}
\item[\hypertarget{G1}{G1.}] $1 = T_{i, 1} < T_{i, 2} < \ldots < T_{i, k} < T_{i, k + 1} \leq (n - i + 1)k + 1$ for all $i$, and
\item[\hypertarget{G2}{G2.}] $T_{i, j} \leq T_{i + 1, j} + j - 1$ for all $i, j$.
\end{description}

\begin{figure}[t]
\begin{center}
\input{goodtableau.tex}
\end{center}
\caption{A good tableau in English notation, i.e., with the first row on top.}
\end{figure}

Given a permutation $w \in \mathcal{L}_{n, k}(1\cdots (k + 2))$, we construct a good tableau $T(w)$ recursively as follows: if $n = 1$ then $w = 12\cdots k$ and we define $T(w)$ by $T(w)_{1, j} = j$.  If $n > 1$, for $1 \leq j \leq k + 1$, let $T_{1, j} : = T(w)_{1, j}$ be the smallest entry of $w$ that is the largest entry in a $j$-term increasing subsequence, or $(n - i + 1)k + 1$ if there is no such entry.  Let $w'$ be the permutation that is order-isomorphic to $w$ with its last $k$ entries removed.  Then we fill the rest of the tableau with $T(w')$, so $T(w)_{i + 1, j} = T(w')_{i, j}$ for $1 \leq i \leq n - 1$.  We note that $T_{1, j}$ could be equivalently defined as the last-occurring entry of $w$ that is the largest term in a $j$-term increasing subsequence of $w$ but is not the largest term in a $(j + 1)$-term increasing subsequence.

\begin{figure}[ht]
\begin{center}
\input{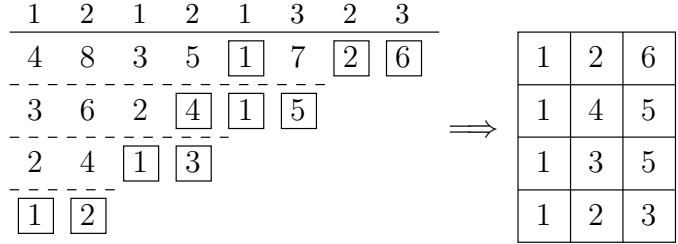}
\end{center}
\caption{An illustration of the application of $T$ to the permutation $w = 48351726 \in \mathcal{L}_{4,2}(1234)$.  For each entry of $w$, the top row indicates the length of the longest increasing subsequence terminating at that entry.}
\label{applyT}
\end{figure}

\begin{prop} \label{alg1}
The algorithm just described is a well-defined map from $\mathcal{L}_{n, k}(1\cdots(k+2))$ to the set of good tableaux of shape $\langle(k + 1)^n\rangle$.
\end{prop}

\begin{proof}
The first entry of the first row of $T(w)$ is the smallest entry in $w$, so $T(w)_{1,1} = 1$.  Now note that the largest term in a $(j + 1)$-term increasing subsequence in $w$ is strictly larger than the term of the subsequence that precedes it, and this term is the largest term of a $j$-term increasing subsequence of $w$.  Thus for each $j$, the smallest entry of $w$ that is the largest term in a $j$-term increasing subsequence is strictly smaller than the smallest entry of $w$ that is the largest term in a $(j + 1)$-term increasing subsequence, assuming the latter exists.  Also, if $n \geq 1$ then every element of $\mathcal{L}_{n, k}$ contains some entry (viz., the $k$th) that is the largest element of a $k$-term increasing subsequence, so at most one copy of $nk + 1$ can appear in the first row of our tableau and it can only appear as the last entry in the row.  Thus, the first row of $T := T(w)$ consists of $k + 1$ distinct values appearing in increasing order, the smallest of which is $1$ and the largest of which is no larger than $nk + 1$.  It follows that condition \hyperlink{G1}{G1} is satisfied for the first row of $T$, so by induction it is satisfied for all rows of $T$.

Because of the recursive nature of the construction, to establish \hyperlink{G2}{G2} for $T$ it suffices to show that $T_{1, j} \leq T_{2, j} + j - 1$ for each $j$.  The number $T_{2, j}$ on the right-hand side of this inequality is equal to $T'_{1, j} := T(w')_{1, j}$, the smallest element of $w'$ that is the largest term in a $j$-term increasing subsequence of $w'$.  Let $a$ be the entry of $w$ in the same position that $T'_{1, j}$ is in $w'$, i.e., the entry that becomes $T'_{1, j}$ after order-isomorphism.  The $k$ entries of $w$ that we remove while forming $w'$ fall into three sets: 
\begin{itemize}
\item those strictly less than $T_{1, j}$, of which there are at most $j - 1$,
\item those in the interval $[T_{1, j}, a)$, of which there are at most $a - T_{1, j}$, and
\item those that are larger than $a$.
\end{itemize}
The difference $a - T'_{1, j}$ is exactly the number of values less than $a$ that we remove while going from $w$ to $w'$, so it is the sum of the sizes of the first two sets above.  Thus
\[
a - T'_{1, j} \leq j - 1 + a - T_{1, j}
\]
and so $T_{1, j} \leq T'_{1, j} + j - 1 = T_{2, j} + j - 1$, as desired.  It follows that our algorithm maps every permutation in $\mathcal{L}_{n, k}(12\cdots(k + 2))$ to a good tableau.  
\end{proof}
Figure \ref{applyT} shows how to apply the map $T$ to the permutation $43851726 \in \mathcal{L}_{4, 2}(1234)$.

We now give a recursive map from good tableaux of shape $\langle(k + 1)^n \rangle$ to $\mathcal{L}_{n, k}(1 \cdots (k + 2))$.  There is only one good tableau $T$ with a single row of length $k + 1$ (the one satisfying $T_{1, j} = j$) and for this tableau we define $w(T) = 12\cdots k$.  Otherwise, for a good tableau $T$ with more than one row, let $m$ be the largest index such that $T_{1, m} = T_{2, m} + m - 1$ and let $T'$ be the tableau that results when we remove the first row of $T$.  Then define $w := w(T)$ by setting $w_{n,1}w_{n,2}\cdots w_{n,k}= T_{1,1}\cdots T_{1, m - 1}T_{1, m + 1} \cdots T_{1, k + 1}$ and choosing the earlier values of $w$ so that the first $(n - 1)k$ values of $w$ are order-isomorphic to $w(T')$.  We say that $T_{1, m}$ is \emph{omitted} from $w$ and that the other entries of the first row of $T$ are \emph{sent to} entries of $w$, and we use the same terminology for entries of larger-numbered rows.  We must show that this process is well-defined, that its output is a permutation in $\mathcal{L}_{n, k}(1\cdots (k + 2))$ and that it really is the inverse of our previous map.  

\begin{prop} \label{alg2}
The algorithm just described is a well-defined function from the set of good tableaux of shape $\langle(k + 1)^n\rangle$ to $\mathcal{L}_{n, k}(1\cdots(k+2))$.
\end{prop}

\begin{proof}In order to show that this process is well-defined, we need to know that a value $m$ actually exists and that the values $w_{n, 1}, \ldots, w_{n, k}$ that we assign are $k$ distinct elements of $[nk]$.  For the former, note that $T_{1, 1} = 1 = T_{2, 1} + 1 - 1$, so there is at least one index $j$ satisfying the needed condition and thus there is a maximal such index and we really can choose a value $m$.  For the latter, we note that condition \hyperlink{G1}{G1} ensures the values $w_{n, 1}, \ldots, w_{n, k}$ will be distinct (in fact, that they will occur in increasing order) and so by induction on $n$ we need only show that each of them lies in $[nk]$.  The only entry of the first row of $T$ that might not lie in $[nk]$ is $T_{1, k + 1}$, provided $T_{1, k+1} = nk + 1$.  In this case, we have by \hyperlink{G1}{G1} and \hyperlink{G2}{G2} that
\[
(n - 1)k + 1 = T_{1, k + 1} - (k + 1) + 1 \leq T_{2, k + 1} \leq (n - 1)k + 1, 
\]
so these inequalities must be equalities and $T_{1, k + 1} - (k + 1) + 1 = T_{2, k + 1}$.  Then $m = k + 1$ and we omit the problematic value.  Thus this algorithm is well-defined as a map from good tableaux of shape $\langle (k + 1)^n\rangle$ to permutations in $S_{nk}$.

As noted in the preceding paragraph, the last $k$ values of $w$ will occur in increasing order and so by induction we have that \hyperlink{L1}{L1} holds for $w$.  Thus, to show that $w \in \mathcal{L}_{n, k}(1\cdots (k + 2))$ we are left to check that $w_{n - 1, j + 1} > w_{n, j}$ for $1 \leq j \leq k - 1$, which will imply \hyperlink{L2}{L2} by induction, and that $w$ contains no increasing subsequence of length $k + 2$.  For both statements we will use the following lemma:

\begin{lemma}\label{lemma}
Given a good tableau $T$, if $T_{i, a}$ and $T_{i + 1, b}$ are both sent to entries of $w(T)$ (i.e., neither of them is omitted) and $T_{i + 1, b} \geq T_{i + 1, a}$ then $T_{i + 1, b}$ is sent to a larger value in $w$ than $T_{i, a}$ is.  
\end{lemma}

\begin{proof}[Proof of lemma.]
We may assume without loss of generality that $i = 1$.  Suppose first that $a < m$.  Then the $a - 1$ numbers $T_{1, 1}, \ldots, T_{1, a - 1}$ are all the values smaller than $T_{1, a}$ that are among the last $k$ entries of $w$.  Thus, the prefix of $w$ of length $(n - 1)k$ contains exactly $T_{1, a} - 1 - (a - 1) = T_{1, a} - a$ values less than $T_{1, a}$.  By condition \hyperlink{G2}{G2}, $T_{1, a} - a \leq T_{2, a} - 1$, so the prefix of $w$ of length $(n - 1)k$ contains at most $T_{2, a} - 1$ values less than $T_{1, a}$.  It follows immediately that any entry of the second row of $T$ that is larger than or equal to $T_{2, a}$ and is sent to an entry of $w$ is sent to an entry larger than $T_{1, a}$.  In particular, $T_{2, b}$ is such an entry.

Suppose instead that $a > m$.  In this case there are $a - 2$ values smaller than $T_{1, a}$ among the last $k$ entries of $w$.  Thus, the prefix of $w$ of length $(n - 1)k$ contains exactly $T_{1, a} - 1 - (a - 2) = T_{1, a} - (a - 1)$ values less than $T_{1, a}$.  By \hyperlink{G2}{G2} and the definition of $m$, $T_{1, a} - (a - 1) < T_{2, a}$, so the prefix of $w$ of length $(n - 1)k$ contains fewer than $T_{2, a}$ values less than $T_{1, a}$.  It follows that any entry of the second row of $T$ that is larger than or equal to $T_{2, a}$ and that is sent to an entry of $w$, is sent to an entry larger than $T_{1, a}$.  This completes the proof of the lemma.  \end{proof}

That $w \in \mathcal{L}_{n, k}$ now follows easily: we know that $w_{n, j}$ is the image of (in fact, is equal to) either $T_{1, j}$ (if $j < m$) or $T_{1, j + 1}$ (if $j \geq m$) and similarly that $w_{n - 1, j + 1}$ is the image of either $T_{2, j + 1}$ or $T_{2, j + 2}$.  In all four cases we can apply Lemma \ref{lemma} and \hyperlink{G1}{G1} to conclude that $w_{n - 1, j + 1} > w_{n, j}$, as desired.

Finally, we must verify that $w$ has no $(k + 2)$-term increasing subsequence.  Choose any $k + 2$ entries of $w$.  Each of these is the image under our algorithm of some entry in $T$.  Since $T$ has $k + 1$ columns, two of these entries must come from the same column.  To show that $w$ has no $(k + 2)$-term increasing subsequence, it is enough to show that of these two entries of $T$, the one with a smaller row number is sent to a smaller value in $w$ than the one with the larger row number.  (This is because we build $w$ so that entries of $T$ with larger row numbers get sent to earlier entries of $w$.)  Then it suffices to show that if two entries of $T$ in the same column are both sent to entries of $w$ and all the entries lying between them in the same column are omitted then the entry with larger row number gets sent to a larger value in $w$, and without loss of generality we may take one of these entries to lie in the first row of $T$.  So suppose that $T_{1, j}$ and $T_{i, j}$ get sent to entries of $w$ while for $1 < \ell < i$, $T_{\ell, j}$ is omitted.  If $j < m$ we have 
\begin{equation} \label{weak}
T_{1, j} \leq T_{2, j} + j - 1 = T_{3, j} + 2(j - 1) = \cdots = T_{i, j} + (i - 1)(j - 1).
\end{equation}
By repeated applications of Lemma \ref{lemma} and \hyperlink{G1}{G1} we have that $T_{i, j}$ is sent to a larger value in $w$ than any of the $(i - 1)(j - 1)$ entries $T_{a, b}$ with $1 \leq b < j - 1$ and $1 \leq a \leq i - 1$ that are sent to a value in $w$.  Thus, after we have taken $i - 1$ steps of our algorithm (constructed the last $(i - 1)k$ values of $w$), we have used at least $(i - 1)(j - 1) + 1$ different values less than or equal to $T_{i, j}$.  Thus, the image of $T_{i, j}$ in $w$ is at least $T_{i, j} + (j - 1)(i - 1) \geq T_{1, j}$, as desired.  If instead $j > m$, the argument carries through identically except that we replace the ``$\leq$'' in Equation \ref{weak} with ``$<$'' and we replace ``$(i - 1)(j - 1)$'' with ``$(i - 1)(j - 1) - 1$.''
\end{proof}

Figure \ref{Applyw} shows how to apply the map $w$ to a good tableau of shape $\langle 3^4 \rangle$.

\begin{figure}
\begin{center}
\scalebox{.95}{\input{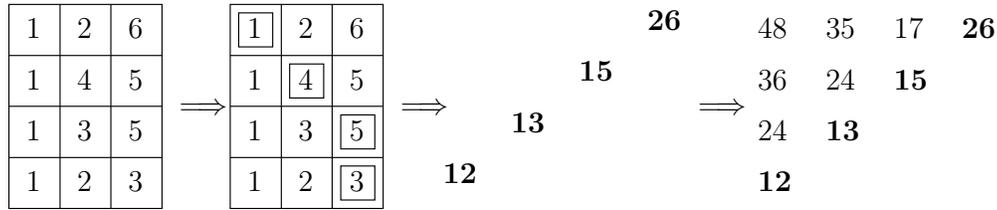}}
\end{center}
\caption{An illustration of the application of the bijection $w$ to a good tableau of shape $\langle 3^4 \rangle$.  The omitted entries of the tableau are boxed.  The entries sent to entries of $w$ are in bold, and we construct the permutation from bottom to top by successively choosing its earlier values to be order-isomorphic to the row below them.  The image of this tableau under $w$ is the top row of the final stage, 48351726.}
\label{Applyw}
\end{figure}

Now that we have described one map going in each direction between the two sets, we may establish our theorem by proving that these maps are inverses.

\begin{prop} The algorithms we have described are mutually-inverse bijections between $\mathcal{L}_{n, k}(1\cdots(k + 2))$ and the set of good tableaux of shape $\langle(k + 1)^n\rangle$.\end{prop}

\begin{proof}
Because of the recursive nature of both constructions, it suffices to check that for a given permutation $p \in \mathcal{L}_{n,k}(1\cdots(k + 2))$, $w(T(p))$ has the same last $k$ entries as $p$ and that for a given good tableau $G$, $T(w(G))$ has the same first row as $G$.  

Choose $p \in \mathcal{L}_{n,k}(1\cdots(k+2))$.  If $p$ contains no $(k + 1)$-term increasing subsequence then the first row of $T := T(p)$ is $p_{n,1}p_{n,2}\cdots p_{n,k}(nk+1)$.  Then as we showed in the first paragraph of the proof of Proposition \ref{alg2}, the last entry $nk + 1$ of the first row is omitted when going from $T$ to $w(T)$ and so the last $k$ entries of $w(T)$ are $p_{n,1}\cdots p_{n,k}$, as desired.  Otherwise, there is some $j \in [k + 1]$ such that the smallest entry of $p$ that is the largest term in a $j$-term increasing subsequence exists but is not among the last $k$ entries of $p$.  As we noted just before the statement of Proposition \ref{alg1}, for each $\ell \in [k + 1]$, the smallest entry of $p$ that is the largest entry in an $\ell$-term increasing subsequence is also the last-occurring entry of $p$ that is the largest entry in an $\ell$-term increasing subsequence but not an $(\ell + 1)$-term increasing subsequence.  Since the last $k$ entries of $p$ occur in increasing order, the maximum lengths of increasing subsequences terminating at these entries are all different.  Then for each $\ell \in [k + 1] \setminus \{j\}$, the smallest entry that is the largest entry in an $\ell$-term increasing subsequence of $p$ lies among the last $k$ entries in $p$, so these last $k$ entries are exactly $T_{1, 1}, \ldots, T_{1, j - 1}, T_{1, j + 1}, \ldots, T_{1, k + 1}$.  

The smallest entry of $p$ that is the largest term in a $j$-term increasing subsequence is equal to $T_{1, j}$, and after we remove the last $k$ entries of $p$ and perform an order-isomorphism it is decreased by exactly $j - 1$ (the smallest $j - 1$ values of the $k$ we remove).  This decreased copy is now the smallest entry of $p'$ that is the largest entry in a $j$-term increasing subsequence of $p'$, so $T_{2, j} = T_{1, j} - j + 1$.  Thus, if we show that for each $\ell > j$ we have $T_{2, \ell} > T_{1, \ell} - \ell + 1$, it will follow that $T_{1, j}$ is the entry omitted from the first row of $T$ when forming $w(T)$, and this will give us our result.

Let $\ell > j$ and let $a$ be the entry of $p$ that becomes $T_{2, \ell}$ after order-isomorphism.  As in the second paragraph of the proof of Proposition \ref{alg1}, we divide the last $k$ entries of $p$ (those removed after forming the first row of $T$) into three categories.  Since $T_{1, j} < T_{1, \ell}$ is not among these last $k$ entries, at most $\ell - 2$ of them are less than $T_{1, \ell}$.  Then there are at most $a - T_{1, \ell}$ entries among the last $k$ that fall in the interval $[T_{1, \ell}, a)$, while the remaining entries are all larger than $a$.  Thus, we remove a total of at most $(a - T_{1, \ell}) + (\ell - 2)$ entries of $p$ that are less than $a$.  The difference $a - T_{2, \ell}$ is exactly the number of entries less than $a$ we remove before applying the order-isomorphism, so 
\[
a - T_{2, \ell} \leq a - T_{1, \ell} + \ell - 2
\]
and thus $T_{1, \ell} - \ell + 1 \leq T_{2, \ell} - 1 < T_{2, \ell}$, as needed.  This demonstrates that $T_{1, j}$ is omitted from $w(T)$, so the last $k$ entries of $w(T)$ are exactly the same as the last $k$ entries of $p$ and so by induction $w \circ T$ is the identity of $\mathcal{L}_{n,k}(1\cdots(k+2))$.  We now proceed to the reverse direction.

Let $G$ be a good tableau.  Suppose first that $G_{1, k + 1} = nk + 1$.  Then by \hyperlink{G1}{G1} and \hyperlink{G2}{G2}, $G_{i, k + 1} = (n - i + 1)k + 1$ for each $i$ and these are exactly the omitted values.  Now apply Lemma \ref{lemma} for $b = a$ and we see that $w(G)$ contains no $(k + 1)$-term increasing subsequence.  (In fact, in this case our bijection is essentially equivalent to that of Proposition \ref{k+1} with some additional order-isomorphisms involved.)  Then the last row of $T(w)$ is exactly the last row of $G$, as desired.  

Now we induct on $n$.  For $n = 1$, $G_{1, k + 1} = k + 1$ and this case was covered in the preceding paragraph.  Otherwise, $G$ has at least two rows and the last $k$ entries of $w := w(G)$ are $G_{1,1}G_{1,2}\cdots G_{1,m - 1}G_{1,m + 1}\cdots G_{1,k + 1}$ for some $m$ with $1 \leq m \leq k + 1$, $G_{1, m} \leq nk$.  In particular, $G_{1, m}$ is an entry of $w$.  We know that each of the entries $G_{1,1}, \ldots, G_{1, m - 1}, G_{1, m + 1}, \ldots, G_{1, k + 1}$ of $w$ is the last entry of $w$ that is the largest term in a $j$-term increasing subsequence but not a $(j + 1)$-term increasing subsequence for some value of $j$, and that these $k$ values of $j$ are distinct.  Then if we can show that $G_{1, m}$ is the smallest entry of $w$ that is the last term in an $m$-term increasing subsequence but not an $(m + 1)$-term increasing subsequence, it will follow that the first row of $T(w)$ is the same as the first row of $G$ and we'll be done by induction.

Note that $G_{1, m}$ cannot be the largest term in an $(m + 1)$-term increasing subsequence of $w$ because if it were, we could append the $(k + 1- m)$-term subsequence $G_{1, m + 1}, \ldots, G_{1, k + 1}$ to the $(m + 1)$-term increasing subsequence ending at $G_{1, m}$ to get a $(k + 2)$-term increasing subsequence of $w$, a contradiction with the $12\cdots (k + 2)$-avoidance of $w$.

Let $w'$ be the permutation order-isomorphic to $w$ with its last $k$ entries removed.  By the inductive hypothesis, $G_{2, m} = G_{1, m} - m + 1$ is the smallest entry of $w'$ that is the largest term in an $m$-term increasing subsequence of $w'$.  When going from $w$ to $w'$ we remove exactly $m - 1$ values less than $G_{1, m}$ (and not $G_{1, m}$ itself) from $w$, so the entry equal to $G_{1, m}$ in $w$ is in the same position that the entry equal to $G_{1, m} - m + 1 = G_{2, m}$ is in $w'$.  Thus, $G_{1, m}$ is the smallest of the first $(n - 1)k$ entries of $w$ that is the largest term in an $m$-term increasing subsequence.  Then $G_{1, m}$ is the smallest entry of $w$ that is the largest term in an $m$-term increasing subsequence \emph{unless} $G_{1, m - 1}$ is the largest term in an $m$-term increasing subsequence of $w$.  

So suppose for sake of contradiction that this is the case.  Then there is some minimal $j \leq m - 1$ such that $G_{1, j}$ is the largest term in a $(j + 1)$-term increasing subsequence in $w$.  Let $a$ be the smallest entry among the first $(n - 1)k$ entries of $w$ that is the largest term in a $j$-term increasing subsequence of $w$.  Then $a < G_{1, j}$.  When we go from $w$ to $w'$, we remove $j - 1$ values less than $G_{1, j}$, and these $j - 1$ values must also be less than $a$.  (Otherwise, we use the $j$-term sequence ending at $a$ together with the increasing subsequence $G_{1, j - 1}G_{1, j}\cdots G_{1, m - 1}G_{1, m + 1}\cdots G_{1, k + 1}$ to give a $(k + 2)$-term increasing subsequence in $w$, a contradiction.)  Thus, the image of $a$ under order-isomorphism is $a - j + 1 < G_{1, j} - j + 1 \leq G_{2, j}$ and is the largest term in a $j$-term increasing subsequence of $w'$.  However, by the inductive hypothesis we have that $G_{2, j}$ itself is the smallest entry of $w'$ that is the largest term in a $j$-term increasing subsequence, a contradiction.  Thus for $j < m$, $G_{1, j}$ is the largest term in a $j$-term increasing subsequence but not a $(j + 1)$-term increasing subsequence, and in particular $G_{1, m - 1}$ is not the largest term in an $m$-term increasing subsequence.  Then since $G_{1, m}$ is the smallest entry of $w$ that is the largest term in an $m$-term increasing subsequence, the first row of $T(w)$ is exactly the same as the first row of $G$.  By induction $T(w(G)) = G$, so $T \circ w$ is the identity on good tableaux.
\end{proof}

We have now shown that the operations $w$ and $T$ are inverse to each other, so they are bijections between $\mathcal{L}_{n, k}(1\cdots (k + 2))$ and the set of good tableaux of shape $\langle(k + 1)^n \rangle$.  To finish the proof of Theorem \ref{k+2thm} we need only the following result:

\begin{prop} \label{tabthm}
There is a bijection between standard Young tableaux and good tableaux of the same shape $\langle(k + 1)^n \rangle$.
\end{prop}

\begin{proof}
To go from a standard Young tableau $T$ to a good tableau $G := G(T)$ of the same shape, take the first row of $G$ to be equal to the first row of $T$.  Let $T'$ be the standard Young tableau that we get by removing the first row of $T$ and applying an order-isomorphism, and let the remaining rows of $G$ be equal to those of $G(T')$.  It's easy to see that the first row of $G$ is increasing and that its minimal entry is $1$.  In addition, the maximal entry in the first row of $G$ is the maximal entry in the first row of $T$, which has $n - 1$ entries below it, and therefore larger than it, in $T$.  Thus we have that this entry is at most $n(k + 1) - (n - 1) = nk + 1$, so we have established \hyperlink{G1}{G1} for $G$ and are left only to check \hyperlink{G2}{G2}.  For any $j$, we have $G_{1, j} = T_{1, j} < T_{2, j}$.  The difference $T_{2, j} - G_{2, j}$ is equal to the number of entries in the first row of $T$ smaller than $T_{2, j}$.  Of these, there are $j - 1$ less than $T_{1, j}$ and at most $T_{2, j} - T_{1, j}$ in the interval $[T_{1, j}, T_{2, j})$, so 
\[
T_{2, j} - G_{2, j} \leq j - 1 + T_{2, j} - T_{1, j} = j - 1 + T_{2, j} - G_{1, j},
\]
so $G_{1, j} \leq G_{2, j} + j - 1$.  Since $G$ is constructed recursively, we have by induction that \hyperlink{G2}{G2} holds and so $G(T)$ is indeed a good tableau.

Now we show how to invert this process.  Given a good tableau $G$ with $n$ rows and $k + 1$ columns, we construct a standard Young tableau $T:= T(G)$.  Let $U_1 = [n(k + 1)]$.  At the $i$th step of our algorithm, for each $j$ we set $T_{i, j}$ to be equal to the $G_{i, j}$th-smallest element of $U_i$, then set $U_{i + 1}$ to be equal to $U_i$ with these $k + 1$ elements removed.  Thus, in the first step we set $T_{1, j} = G_{1, j}$, i.e. the first row of $T$ is equal to the first row of $G$.  We then set $U_2 = U_1 \smallsetminus \{T_{1, 1}, \ldots, T_{1, k + 1}\}$; then $T_{2, 1}$ is the $G_{2, 1} = 1$st smallest element of $U_2$, $T_{2, 2}$ is the $G_{2, 2}$th smallest element of $U_2$, and so on.

We need that this process is well-defined -- in particular, that $U_i$ always has enough elements.  At the $i$th step, the largest element we can request is the $((n - i + 1)k + 1)$th of $U_i$.  But $U_i$ has exactly $(n - i + 1)(k + 1)$ elements and $(n - i + 1)(k + 1) \geq (n - i + 1)k + 1$ because $n \geq i$.  Thus our procedure is well-defined.  

We also need that this procedure always takes as output a standard Young tableau.  By construction and \hyperlink{G1}{G1}, $T$ is a filling of a $(k+1)^n$ rectangle with $[n(k+1)]$ that increases along rows, so we only need that it increases along columns.  By the nature of the construction and the properties of good tableaux, it suffices to show this for the first row.  Note that $T_{1, 1} = 1$, so surely $T_{2, 1} > T_{1, 1}$.  Suppose $T_{2, j} > T_{1, j}$.  Then $T_{2, j + 1} > T_{2, j} > T_{1, j}$ as well, so at least $j$ entries of the first row of $T$ are smaller than $T_{2, j + 1}$.  Note that the difference $T_{2, j + 1} - G_{2, j + 1}$ is exactly the number of entries of the first row of $T$ that are smaller than $T_{2, j + 1}$, so this difference is at least $j$ and thus 
\[
T_{2, j + 1} - T_{1, j + 1} \geq (G_{2, j + 1} + j) - G_{1, j + 1} = j + (G_{2, j + 1} - G_{1, j + 1}) \geq 0,
\]
where the last inequality follows from \hyperlink{G2}{G2}.  Since $T_{2, j + 1} \neq T_{1, j + 1}$ by construction, we have $T_{2, j + 1} > T_{1, j + 1}$, so $T$ in increasing along columns by induction on $j$.

Finally, we have to show that these two operations are inverses.  This is relatively straight-forward: when going from a standard Young tableau to a good tableau, we send an entry in row $i$ to its rank among entries in rows numbered $i$ or larger and when going from a good tableau to a standard Young tableau, we send an entry $\ell$ in row $i$ to a value that is the $\ell$th largest among values not used in rows $1$ through $i - 1$.  These descriptions make the inverse property self-evident.\end{proof}

One direction of the bijection is illustrated in Figure \ref{GtoT} and the other in Figure \ref{TtoG}.

At last, we have given bijections between $\mathcal{L}_{n, k}(1\cdots(k + 2))$ to the set of good tableaux of shape $\langle(k + 1)^n \rangle$ and between this latter set and the set of standard Young tableaux of the same shape, so we may conclude:

\begin{recap}
There is a bijection between $\mathcal{L}_{n, k}(12\cdots (k + 1)(k + 2))$ and the set of standard Young tableaux of shape $\langle(k + 1)^n \rangle$ and so 
\[
|\mathcal{L}_{n, k}(12\cdots (k + 1)(k + 2))| = f^{\langle (k + 1)^n \rangle}.
\]
\end{recap} 

\begin{figure}[t]
\begin{center}
\scalebox{.9}{\input{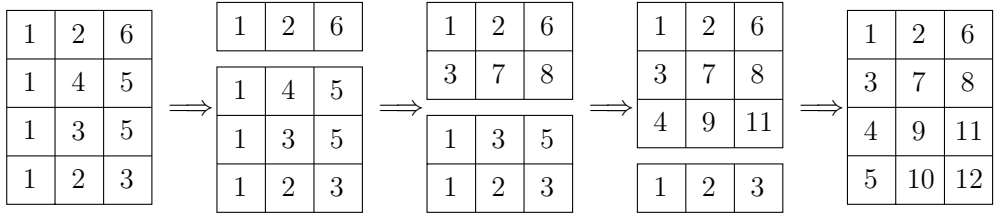}}
\end{center}
\caption{The application of our algorithm sending a good tableau to its corresponding standard Young tableau.  We have $U_1 = [12]$, $U_2 = \{3, 4, 5, 7, 8, 9, 10, 11, 12\}$, $U_3 =  \{4, 5, 9, 10, 11, 12\}$ and $U_4 =  \{5, 10, 12\}$.}
\label{GtoT}
\end{figure}

\begin{figure}[th]
\begin{center}
\scalebox{.9}{\input{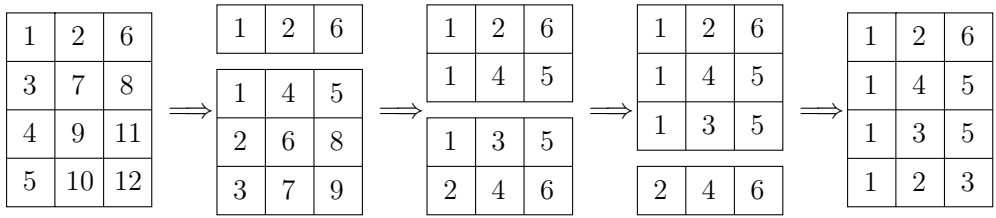}}
\end{center}
\caption{The application of our algorithm sending a standard Young tableau to its corresponding good tableau.}
\label{TtoG}
\end{figure}

\subsection{A second bijection, with a modified version of RSK}\label{subsec:2ndbij}

In this section, we reprove Theorem \ref{k+2thm} using a modification of the RSK insertion algorithm.  Our modification is an extension of the bijection devised by Ouchterlony in \cite{ouch}.  Recall that the RSK insertion algorithm is a map between $S_n$ and the set of pairs $(P, Q)$ of standard Young tableaux such that $\sh(P) = \sh(Q) \vdash n$ with the following properties:

\begin{thm}[{\cite[7.11.2(b)]{EC2}}]\label{LtoR}
If $P$ is a standard Young tableau and $j < k$ then the insertion path of $j$ in $P \leftarrow j$ lies strictly to the left of the insertion path of $k$ in $(P \leftarrow j)\leftarrow k$, and the latter insertion path does not extend below the former.
\end{thm}
\begin{thm}[{\cite[7.23.11]{EC2}}]\label{incsub}
If $w \in S_n$ and $w \overset{\textrm{RSK}}{\longrightarrow} (P, Q)$ with $\sh(P) = \sh(Q) = \lambda$, then $\lambda_1$ is the length of the longest increasing subsequence in $w$.
\end{thm}

Now we describe a bijection from $\mathcal{L}_{n, k}(12 \cdots (k + 2))$ to pairs $(P, R)$ of standard Young tableaux such that $P$ has $nk$ boxes, $R$ has $n$ boxes, and the shape of $R$ can be rotated $180 ^ \circ$ and joined to the shape of $P$ to form a rectangle of shape $\langle(k + 1)^n \rangle$.  (In other words, $\sh(P)'_i + \sh(R)'_{k + 2 - i} = n$ for $1 \leq i \leq k + 1$.)  Observe that the set of such pairs of tableaux is in natural bijection with the set of standard Young tableaux of shape $\langle(k + 1)^n \rangle$: given a tableau of shape $\langle(k + 1)^n \rangle$, break off the portion of the tableau filled with $nk + 1, \ldots, n(k + 1)$, rotate it $180^\circ$ and replace each value $i$ that appears in it with $nk + n + 1 - i$.

Given a permutation $w = w_{1, 1}w_{1, 2} \cdots w_{1, k}w_{2, 1} \cdots w_{n, k}$, let $P_0 = \emptyset$ and let $P_i = ( \cdots ((P_{i - 1}\leftarrow w_{i, 1}) \leftarrow w_{i, 2}) \cdots) \leftarrow w_{i, k}$ for $1 \leq i \leq n$.  Define $P := P_n$, so $P$ is the usual RSK insertion tableau for $w$.  Define $R$ as follows: set $R_0 = \emptyset$ and $\lambda_i = \sh(P_i)$.  Observe that by Theorem \ref{LtoR}, $\lambda_i / \lambda_{i - 1}$ is a horizontal strip of size $k$ and that by Theorem \ref{incsub}, $\lambda_i / \lambda_{i - 1}$ stretches no further right than the $(k + 1)$th column.  Thus there is a unique $j$ such that $\lambda_i / \lambda_{i- 1}$ has a box in the $\ell$th column for all $\ell \in [k + 1]\setminus \{j\}$.  Let $R_i$ be the shape that arises from $R_{i - 1}$ by adding a box filled with $i$ in the $(k + 2 - j)$th column, and let $R := R_n$.  This map is illustrated in Figure \ref{RSK}.

\begin{figure}
\begin{center}
\input{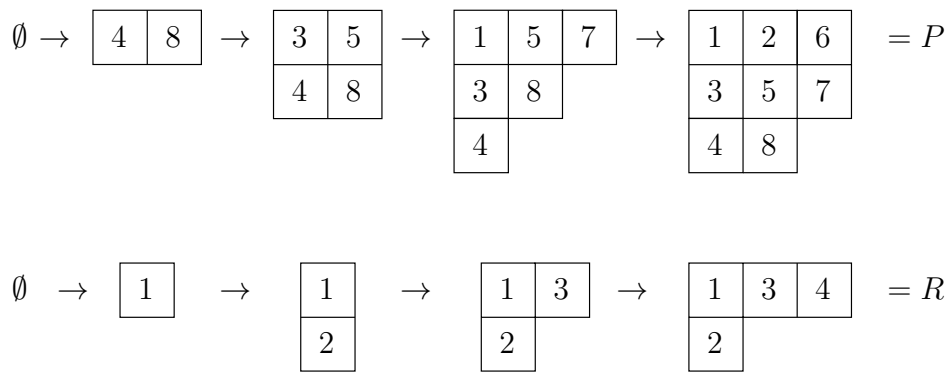}
\end{center}
\caption{An application of our modified version of RSK to the permutation $48351726 \in \mathcal{L}_{4, 2}(1234)$.  Note that only every other insertion step is shown in the construction of $P$.}
\label{RSK}
\end{figure}

\begin{prop}\label{RSKthm}
The algorithm just described is a bijection between the set $\mathcal{L}_{n, k}(12 \cdots (k + 2))$ and pairs $(P, R)$ of standard Young tableaux such that $P$ has $nk$ boxes, $R$ has $n$ boxes, and $\sh(R)$ can be rotated and joined to $\sh(P)$ to form a rectangle of shape $\langle(k + 1)^n \rangle$.
\end{prop}

\begin{proof}
By construction, it is clear that $P$ is a standard Young tableau with $nk$ boxes and that $R$ is a shape with $n$ boxes filled with $[n]$ such that rotating $R$ by $180^\circ$ we may join it to $P$ in order to get a rectangle of shape $\langle (k + 1)^n \rangle$.  So it is left to show that $R$ is actually a standard Young tableau and that this process is a bijection.

The tableau $R$ is constructed in such a way that it is automatically increasing down columns.  Also, $R$ increases across rows unless at some intermediate stage it is not of partition shape, i.e., for some $i, j$ we have that $R_i$ contains more boxes in column $j + 1$ than in column $j$.  But in this case it follows from the construction of $R$ that $P_i$ has more boxes in column $k + 2 - j$ than in column $k + 1 - j$.  Since we have by properties of RSK that every $P_i$ is of partition shape, this is absurd, and so $R$ really is a standard Young tableau.

Finally, we need that this algorithm is invertible and that its inverse takes pairs of tableaux of the given sort to permutations with the appropriate restrictions.  Invertibility is immediate, since from a pair $(P, R)$ of standard Young tableaux of appropriate shapes we can construct a pair of standard Young tableaux $(P, Q)$ of the same shape with $w \mapsto (P, R)$ under our algorithm exactly when $w \overset{\textrm{RSK}}{\longrightarrow} (P, Q)$: if $R$ has entry $i$ in column $j$, place the entries $ki - k + 1, ki - k + 2, \ldots, ki$ into columns $1, \ldots, k + 1- j, k + 3 - j, \ldots, k + 1$ of $Q$, respectively.  Moreover, by Theorem \ref{LtoR} we have that the preimage under RSK of this pair $(P, Q)$ must consist of $n$ runs of $k$ elements each in increasing order, i.e., it must satisfy \hyperlink{L1}{L1}, and by Theorem \ref{incsub} it must have no increasing subsequence of length $k + 2$.  Then by the remarks in Section~\ref{defs} following the definition of $\mathcal{L}_{n, k}$ we have that the preimage satisfies \hyperlink{L2}{L2} as well.  This completes the proof.
\end{proof}

\begin{NB}
Observe that Figures \ref{applyT} through \ref{RSK} demonstrate that the bijection of this section is different than the bijection of the preceding subsection.  Moreover, the two are not related by reverse-complementation of permutations and/or rotated-complementation of tableaux.
\end{NB}

\section{Doubly-alternating permutations and their generalization}\label{sec:ouch}

In \cite{ouch}, a variety of enumerative results were obtained for pattern avoidance in \emph{doubly-alternating permutations}.  These are permutations $w$ such that both $w$ and $w^{-1}$ are alternating.\footnote{These permutations have some independent interest; for example, Ira Gessel has conjectured (unpublished) that for fixed $n$, the doubly-alternating permutations are the maximally-sized sets of the form $\beta_n(S, T) = \{w \in S_n \mid \operatorname{Des}(w) = S, \operatorname{Des}(w^{-1}) = T\}$, where $\operatorname{Des}$ gives the descent set of a permutation.  Foulkes \cite{Foulkes} and Stanley \cite{StanAltSym} enumerated doubly-alternating permutations using symmetric functions.} One of these results in particular is similar in flavor to the results we have discussed so far, namely:

\begin{thm}[{\cite[Theorem 5.2]{ouch}}] 
There exists a bijection between the set of doubly alternating permutations of length $2n$ that avoid the pattern $1234$ and $S_n(1234)$.
\end{thm}

We extend this result as follows:

\begin{thm}\label{thm:doublyLnk}
For all $n, k \geq 1$, the set of permutations $w \in \mathcal{L}_{n, k}(1 \cdots (k + 2))$ such that $w^{-1} \in \mathcal{L}_{n, k}(1 \cdots (k + 2))$ is in bijection with $S_n(1\cdots (k + 2))$.
\end{thm}
\begin{proof}
Our proof is nearly identical to that of \cite{ouch}.  Recall the argument of Proposition \ref{RSKthm}: we showed that if $w \in \mathcal{L}_{n, k}(1 \cdots (k + 2))$ and $w \overset{\textrm{RSK}}{\longrightarrow} (P, Q)$ then we can encode $Q$, a tableau with $nk$ boxes, as $R$, a tableau with only $n$ boxes, by replacing a set of boxes in columns $1, \ldots, j - 1, j + 1, \ldots, k + 1$ by a single box in column $k + 2 - j$.  We refer to this as the pairing of the column $k + 2 - j$ in $R$ with the columns $[k+1]\smallsetminus \{j\}$ in $Q$.

Suppose that $w^{-1} \in \mathcal{L}_{n, k}(1 \cdots (k + 2))$.  We have by \cite[Theorem 7.13.1]{EC2} that $w^{-1} \overset{\textrm{RSK}}{\longrightarrow} (Q, P)$.  Then it follows that we can perform the same column-pairing procedure to encode $P$ as a tableau $S$ with $n$ boxes.  We have $\sh(P) = \sh(Q)$, so by the nature of the pairing we have that $\sh(S) = \sh(R)$.  Thus $R$ and $S$ are standard Young tableaux of the same shape, with $n$ boxes and at most $k + 1$ columns.  Moreover, this map is a bijection between $\{w \mid w, w^{-1} \in \mathcal{L}_{n, k}(1 \cdots (k + 2)) \}$ and the set of all such pairs of tableaux.  Finally, RSK is a bijection between the set of these pairs of tableaux and $S_n(1\cdots (k + 2))$, so we have our desired bijection.
\end{proof}

It is possible to exploit various properties of RSK to give other results of a similar flavor, of which we give one example.  If we apply RSK to an involution $w \in \mathcal{L}_{n, k}(1 \cdots (k + 2))$, the result is a pair $(P, P)$ of equal tableaux.  In this case the associated tableaux $R$ and $S$ described in the proof of the preceding theorem are also equal.  Thus the image of $w$ under the bijection of Theorem~\ref{thm:doublyLnk} is also an involution, and moreover, every involution in $S_n(1\cdots (k + 2))$ is the image of some involution in $\mathcal{L}_{n, k}(1 \cdots (k + 2))$ under this map.  This argument implies the following result:

\begin{cor}
The number of involutions in $\mathcal{L}_{n, k}(1 \cdots (k + 2))$ is equal to the number of involutions in $S_n(1\cdots (k + 2))$.
\end{cor}

\section{An analogue of alternating permutations of odd length} \label{sec:oddlength}
The results of Sections \ref{k+1sec} and \ref{sec:k+2} concern $\mathcal{L}_{n, k}$, a generalization of the set $A_{2n}$ of alternating permutations of even length.  We now define a set $\mathcal{L}_{n, k; r}$ that is one possible associated analogue of alternating permutations of odd length.  We briefly describe the changes that need to be made to the bijections of Sections \ref{k+1sec} and \ref{sec:k+2} in order to have them apply in this context and the analogous theorems that result.  The proofs are very similar to our preceding work, and we trust that the interested reader can work out the details for herself.

For $k \geq 1$ and $0 \leq r \leq k - 1$, define $\mathcal{L}_{n, k; r}$ to be the set of permutations $w = w_{0, 2} w_{0, 3} \cdots w_{0, r + 1} w_{1,1} \cdots w_{n,k}$ in $S_{nk + r}$ that satisfy the conditions
\begin{description}
\item[L1$'$.] $w_{i,j} < w_{i,j + 1}$ for all $i, j$ such that $1 \leq i \leq n$, $1 \leq j \leq k - 1$ or $i = 0$, $2 \leq j \leq r$, and
\item[L2$'$.] $w_{i,j + 1} > w_{i+1, j}$ for all $i, j$ such that $1 \leq i \leq n - 1$, $1 \leq j \leq k - 1$ or $i = 0$, $1 \leq j \leq r$.
\end{description}
Note in particular that $\mathcal{L}_{n, k; 0} = \mathcal{L}_{n, k}$ and that $\mathcal{L}_{n, 2; 1}$ is the set $A'_{2n + 1}$ of down-up alternating permutations of length $2n + 1$, i.e., those for which $w_1 > w_2 < w_3 > \ldots < w_{2n + 1}$.  As with $\mathcal{L}_{n, k}$, these permutations are the reading words of tableaux of a certain shape, which is illustrated in Figure~\ref{oddshapes}.  

The following results extend our work in the preceding sections to this context:

\begin{figure}
\begin{center}
\input{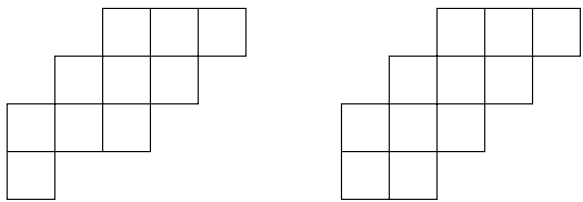}
\end{center}
\caption{The Young diagrams whose standard Young tableaux have reading words $\mathcal{L}_{3, 3; 1}$ and $\mathcal{L}_{3, 3; 2}$.}
\label{oddshapes}
\end{figure}

\begin{prop}
We have a bijection between $\mathcal{L}_{n, k;r}(12\cdots(k + 1))$ and the set of standard Young tableaux of shape $\langle k^n, r\rangle$.
\end{prop}

The bijection of Section \ref{k+1sec} carries over in the obvious way.

\begin{prop}
We have a bijection between $\mathcal{L}_{n, k;r}(12\cdots (k + 2))$ and the set of standard Young tableaux of shape $\langle(k + 1) ^{n - 1}, k, r\rangle$.
\end{prop}

Observe that this coincides with the result of Theorem \ref{k+2thm} in the case $r = 0$, since every standard Young tableau of shape $\langle(k + 1)^{n - 1},k\rangle$ can be extended uniquely to a standard Young tableau of shape $\langle(k + 1)^n \rangle$ by adding a single box filled with $nk + n$.

\begin{proof}[Proof idea 1.]
We describe how to modify the bijection of Section \ref{subsec:1stbij} so that it works in the context of $\mathcal{L}_{n, k;r}(12\cdots (k + 2))$.  To start we redefine the set of good tableaux of shape $\lambda$ (in agreement with Proposition \ref{tabthm}) as the collection of objects that result when we apply the process illustrated in Figure \ref{TtoG} to standard Young tableaux of shape $\lambda$.  That is, a filling $G$ of the diagram of $\lambda$ with positive integers is said to be a good tableau if there is a standard Young tableau $T$ of shape $\lambda$ such that the $(i, j)$th entry of $G$ is equal to $k$ if and only if the $(i, j)$th entry of $T$ is the $k$th smallest among all those entries of $T$ lying in rows numbered $i$ or larger.  

Having made this redefinition, suppose that we are given a permutation $w \in \mathcal{L}_{n, k;r}(12\cdots (k + 2))$.  To find its associated tableau, we perform the first $n - 1$ steps as in the bijection preceding Proposition \ref{alg1}.  When constructing the $n$th row of the tableau $T(w)$, instead of adding $k + 1$ boxes, add only $k$, omitting the last box.  (Observe that this box would otherwise always hold the value $k + r + 2$, since no permutation in $\mathcal{L}_{1, k; r}$ has an increasing subsequence of length $k + 1$.)  Then fill the $(n + 1)$th row with the values $1, 2, \ldots, r$.  Convert the result to a standard Young tableau as described in the proof of Proposition \ref{tabthm} and illustrated in Figure \ref{GtoT}.

To invert this process, suppose we are given a standard Young tableau $T$ of the appropriate shape.  Convert $T$ to a good tableau (under our new definition) by the process described in the proof of Proposition \ref{tabthm} and illustrated in Figure \ref{TtoG}.  Then proceed as in the bijection preceding Proposition \ref{alg2}, with the following variations: no entries should be omitted from the last two rows of the tableau.  For the $(n - 1)$th row (the last row of size $k + 1$), omit the last entry if it is equal to $2k + r + 2$ and otherwise proceed as usual.  (This is equivalent to saying that for the purposes of the $(n - 1)$th row, we pretend that the $n$th row actually is of length $k + 1$, with last entry equal to $k + r + 2$.)  
\end{proof}

\begin{proof}[Proof idea 2.]
We describe how to modify the bijection of Section \ref{subsec:2ndbij} so that it works in the context of $\mathcal{L}_{n, k;r}(12\cdots (k + 2))$.  First, we note that standard Young tableaux of shape $\langle(k + 1)^{n - 1}, k, r\rangle$ are in bijection with pairs $(P, R)$ of tableaux such that $\sh(P)$ is a partition of $nk + r$, $\sh(R)$ is a skew shape $\mu / \langle k + 1 - r, 1 \rangle$ with $n - 1$ boxes, and $\sh(R)$ can be rotated and joined to $\sh(P)$ to form the partition $\langle(k + 1)^{n - 1}, k, r\rangle$.  Our bijection is between $\mathcal{L}_{n, k;r}(12\cdots (k + 2))$ and pairs of tableaux of this form.

Suppose that we are given a permutation $w \in \mathcal{L}_{n, k;r}(12\cdots (k + 2))$.  To build the associated pair of tableaux, we make the following changes to the algorithm given just before the proof of Proposition \ref{RSKthm}: we again make use of RSK, this time with intermediate tableau $P_{-1}, \ldots, P_n = P$ and $R_1, \ldots, R_n = R$.  We set $P_{-1} = \emptyset$ and let $P_{0}$ be the result of inserting the first $r$ values of $w$ into $P_{-1}$.  For $0 \leq i \leq n - 1$, we build $P_{i + 1}$ from $P_i$ by inserting the next $k$ values of $w$.  The $R_i$ are no longer standard Young tableaux but are instead standard skew Young tableaux, with $R_1 = \langle k + 1 - r, 1 \rangle / \langle k + 1 - r, 1\rangle$.  For $1 \leq i \leq n - 1$, we build $R_{i + 1}$ from $R_i$ by the same column-pairing process as before, always preserving the removed shape $\langle k + 1 - r, 1\rangle$.  Thus for each $i$ we have $R_i = \mu / \langle k + 1 - r, 1\rangle$ for some $\mu \vdash k + 1 - r + i$.  (In particular, the first box is added to $R_1$ to form $R_2$ after $2k + r$ values have been inserted into $P$.)

To invert this process, we proceed as in Proposition \ref{RSKthm} for $n - 1$ steps, until $R$ has been exhausted.  At that point, what remains of $P$ will be of shape $\langle k, r\rangle$.  Then we set $w_{0, 2}w_{0, 3}\cdots w_{0, r+1}$ equal to the second row and $w_{1, 1}w_{1, 2} \cdots w_{1, k}$ equal to the first row of (what remains of) $P$.
\end{proof}

As an immediate corollary we have that the number of down-up alternating permutations of length $2n + 1$ avoiding $1234$ is the same as the number of standard Young tableaux of shape $\langle 3^{n - 1}, 2, 1\rangle$.  Taking reverse-complements of permutations gives a bijection between the set of down-up alternating permutations of length $2n + 1$ avoiding $1234$ and the set $A_{2n + 1}(1234)$ of (up-down) alternating permutations of length $2n + 1$ avoiding $1234$, and applying the hook-length formula yields the following result:

\begin{cor}
We have $\displaystyle |A_{2n + 1}(1234)| = \frac{16 (3n)!}{(n - 1)!(n + 1)!(n + 3)!}$ for all $n \geq 1$.
\end{cor}

\section{Reading words of Young tableaux of arbitrary skew shapes}\label{skewsec}

So far, we have considered permutations that arise as the reading words of standard skew Young tableaux of particular nice shapes.  In this section, we expand our study to include pattern avoidance in the reading words of standard skew Young tableaux of \emph{any} shape.  As is the case for pattern avoidance in other settings, it is relatively simple to handle the case of small patterns (in our case, patterns of length three or less), but it is quite difficult to prove exact results for larger patterns.

In addition to encompassing pattern avoidance for permutations of length $n$ (via the shape $\langle n, n - 1, \ldots, 2, 1\rangle / \langle n - 1, n - 2, \ldots, 1\rangle$), alternating permutations (via the shape $\langle n + 1, n , \ldots, 3, 2\rangle / \langle n - 1, n - 2, \ldots, 1\rangle$ and three other similar shapes), and more generally $\mathcal{L}_{n, k}$ for any $k$, this type of pattern avoidance also encompasses the enumeration of pattern-avoiding permutations by descent set (when the skew shape is a ribbon) or with certain prescribed descents (when the shape is a vertical strip).  Thus, on one hand the strength of our results is constrained by what is tractable to prove in these circumstances, while on the other hand any result we are able to prove in this context applies quite broadly.

\begin{NB}\label{NB}
We place the following restriction on all Young diagrams in this section: we will only consider diagrams $\lambda / \mu$ such that the inner (northwestern) boundary of $\lambda/\mu$ contains the entire outer (southeastern) boundary of $\mu$.  For example, the shape $\langle 4, 2, 1\rangle / \langle 2, 1\rangle$ meets this condition, while the shape $\langle 5, 2, 2, 1\rangle / \langle 3, 2, 1\rangle$ does not.  
\end{NB}
Note that imposing this restriction does not affect the universe of possible results, since for a shape $\lambda / \mu$ failing this condition we can find a new shape $\lambda' / \mu'$ that passes it and has an identical set of reading words by moving the various disconnected components of $\lambda / \mu$ on the plane.  For example, for $\lambda / \mu = \langle 5, 2, 2, 1\rangle/\langle 3, 2, 1\rangle$ we have $\lambda' / \mu' = \langle 4, 2, 1\rangle / \langle 2, 1\rangle$ -- just slide disconnected sections of the tableau together until they share a corner.  This is illustrated in Figure \ref{slidereduce}.

\begin{figure}
\begin{center}
\input{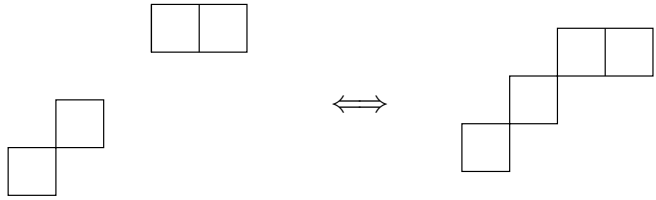}
\end{center}
\caption{Moving separated components gives a new shape but leaves the set of reading words of tableaux unchanged.}
\label{slidereduce}
\end{figure}

\subsection{The patterns $213$ and $132$}\label{213sec}

The equality $|S_n(213)| = |S_n(132)| = C_n$ is a simple recursive result.  In \cite{mans} it was shown that $|A_{2n}(132)| = |A_{2n + 1}(132)| = C_n$ (and so by reverse-complementation also $|A_{2n}(213)| = C_n$), and a bijective proof of this fact with implications for simultaneous avoidance of multiple patterns was given in \cite{self}.  Here we extend this result to the reading words of tableaux of any fixed shape.

\begin{prop}\label{213prop}
The number of tableaux of skew shape $\lambda / \mu$ whose reading words avoid the pattern $213$ is equal to the number of partitions whose Young diagram is contained in that of $\mu$ (subject to Note \ref{NB}).
\end{prop}

Note that this is a very natural $\mu$-generalization of the Catalan numbers: the outer boundaries of shapes contained in $\langle n - 1, n - 2, \ldots, 2, 1\rangle$ are essentially Dyck paths of length $2n$ missing their first and last steps.

\begin{proof} 
We begin with a warm-up and demonstrate the claim in the case that $\mu$ is empty.  In this case, the Proposition states that there is a unique standard Young tableau of a given shape $\lambda = \langle \lambda_1, \lambda_2, \ldots\rangle$ whose reading word avoids the pattern $213$.  In order to see this, we note that the reading word of every straight (i.e., non-skew) tableau ends with an increasing run of length $\lambda_1$ and that the first entry of this run is $1$.  Since this permutation is $213$-avoiding, each entry following the $1$ must be smaller than every entry preceding the $1$ and so this run consists of the values from $1$ to $\lambda_1$.  Applying the same argument to the remainder of the tableau (now with the minimal element $\lambda_1 + 1$), we see that the only possible filling is the one we get by filling the first row of the tableau with the smallest possible entries, then the second row with the smallest remaining entries, and so on.  On the other hand, the reading word of the tableau just described is easily seen to be $213$-avoiding, so we have our result in this case.

For the general case we give a recursive bijection.  This bijection is heavily geometric in nature, and we recommend that the reader consult Figures \ref{fig:213.1} and \ref{fig:213.2} to most easily understand what follows.  Suppose we have a tableau $T$ of shape $\lambda / \mu$ with entry $1$ in position $(i, j)$, an inner corner.  

Divide $T$ into two pieces, one consisting of rows $1$ through $i$ with the box $(i, j)$ removed, the other consisting of rows numbered $i + 1, i + 2$, etc.  Let $T_1$ be the tableau order-isomorphic to the first part and let $T_2$ be the tableau order-isomorphic to the second part.  Working recursively, suppose we have defined our map for smaller tableaux: let $\nu = \langle \nu_1, \ldots, \nu_i\rangle$ be the image of $T_1$ and let $\iota = \langle \iota_1, \iota_2, \ldots\rangle$ be the image of $T_2$.  Then the partition $\tau$ associated to $T$ is given by $\tau = \langle \nu_1 + j, \ldots, \nu_i + j, \iota_1, \iota_2, \ldots\rangle$.  Geometrically, $\tau$ consists of all boxes $(k, l)$ with $k < i$ and $l \leq j$ together with the result of applying our process to the right of this rectangle and the result of applying it below the rectangle, shifted up one row.  By construction, $\tau$ is a partition whose Young diagram fits inside $\mu$.

\begin{figure}
\begin{center}
\scalebox{.8}{\input{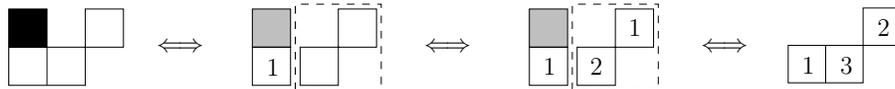}}
\end{center}
\caption{Our bijection applied to the pair $(\langle 3, 2\rangle / \langle 2\rangle, \langle 1\rangle)$ to generate a standard Young tableau.}
\label{fig:213.1}
\end{figure}

\begin{figure}
\begin{center}
\scalebox{.65}{\input{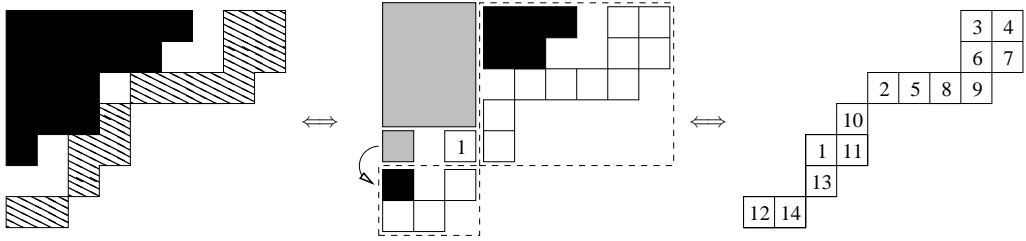}}
\end{center}
\caption{A partial example: the standard Young tableau at right corresponds to the pair $(\langle 9, 9, 8, 4, 4, 3, 2\rangle / \langle 7, 7, 4, 3, 2, 2\rangle, \langle 6, 5, 3, 3, 1\rangle)$ under our bijection.}
\label{fig:213.2}
\end{figure}

To invert this process, start with a pair $(\lambda/\mu, \tau)$ such that $\tau$ fits inside $\mu$.  Let $i$ be the largest index such that $\tau_{i-1} > \mu_i$, or let $i = 1$ if no such index exists.  We divide $\tau$ and $\lambda/\mu$ into two pieces.  To split $\tau$, we delete the boxes that belong to the rectangle of shape $\langle (\mu_i + 1)^{i - 1} \rangle$, leaving a partition of shape $\nu = \langle \tau_1 - \mu_i - 1, \tau_2 - \mu_i - 1, \ldots, \tau_{i - 1} - \mu_i - 1\rangle$ to the right of the rectangle and a partition of shape $\iota = \langle \tau_i, \tau_{i + 1}, \ldots\rangle$ below the rectangle.  To split $\lambda/\mu$, we begin by filling the box $(i, \mu_i + 1)$ with the entry $1$.  Then we take the boxes to the right of this entry as one skew shape $\alpha_1 / \beta_1 = \langle \lambda_1 - \mu_i - 1, \lambda_2 - \mu_i - 1, \ldots, \lambda_i - \mu_i - 1\rangle / \langle \mu_1 - \mu_i - 1, \mu_2 - \mu_i - 1, \ldots, \mu_{i - 1} - \mu_i - 1\rangle$ and the boxes below it as our second skew shape $\alpha_2 / \beta_2 = \langle \lambda_{i + 1}, \lambda_{i + 2}, \ldots\rangle / \langle \mu_{i + 1}, \mu_{i + 2}, \ldots \rangle$. 
 Note that by the choice of $i$, $\nu$ fits inside $\beta_1$ while $\iota$ fits inside $\beta_2$.  Thus, we can apply the construction recursively with the pairs $(\alpha_1 / \beta_1, \nu)$ and $(\alpha_2 / \beta_2, \iota)$, but we fill $\alpha_1 / \beta_1$ with the values $2, \ldots, |\alpha_1/\beta_1| + 1$ and we fill $\alpha_2 / \beta_2$ with the values $|\alpha_1 / \beta_1| + 2, \ldots, |\lambda / \mu| = |\alpha_1 / \beta_1| + |\alpha_2 / \beta_2| + 1$.  (Observe that this coincides with what we did in the first paragraph for $\mu = \emptyset$.)  By induction, this gives us a standard skew Young tableau of shape $\lambda / \mu$.  Note that the reading word of this tableau can be decomposed as the reading word of $\alpha_2 / \beta_2$ followed by the entry $1$ followed by the reading word of $\alpha_1 / \beta_1$.  By the recursive nature of the construction we may assume that the reading words of $\alpha_1 / \beta_1$ and $\alpha_2 / \beta_2$ are $213$-avoiding, and we chose the entries of $\alpha_2 / \beta_2$ to be all larger than the entries of $\alpha_1 / \beta_1$, so it follows that the reading word of the resulting tableau is $213$-avoiding.  Since the two maps we've described are clearly inverse to each other, we have that they are the desired bijections.
\end{proof}

\begin{cor}
We have $|\mathcal{L}_{n, k; r}(213)| = C_n$ for all $n \geq 1$ and $k - 1 \geq r \geq 0$.
\end{cor}

Note that knowing the number of tableaux of each shape whose reading words avoid $213$ automatically allows us to calculate the number of tableaux of a given shape whose reading words avoid $132$: if $\lambda = \langle \lambda_1, \ldots, \lambda_k\rangle$ and $\mu$ is contained in $\lambda$, the rotated complement operation $T \mapsto T^*$ defined in Section \ref{defs} is a bijection between tableaux of shape $\lambda / \mu$ and tableaux of shape $\langle \lambda_1 - \mu_k, \lambda_1 - \mu_{k - 1}, \ldots, \lambda_1 - \mu_1\rangle / \langle \lambda_1 - \lambda_k, \lambda_1 - \lambda_{k - 1}, \ldots, \lambda_1 - \lambda_2\rangle$.  Moreover, the reading word of $T^*$ is exactly the reversed-complement of the reading word of $T$.  It follows that the reading word of $T$ avoids $132$ if and only if the reading word of $T^*$ avoids $213$.  This argument establishes the following corollary of Proposition \ref{213prop}:

\begin{cor}\label{132prop}
The number of tableaux of skew shape $\lambda /\mu$ whose reading words avoid the pattern $132$ is equal to the number of partitions whose Young diagram is contained in that of the partition $\langle \lambda_1 - \lambda_k, \lambda_1 - \lambda_{k - 1}, \ldots, \lambda_1 - \lambda_2\rangle$ (subject to Note \ref{NB}).
\end{cor}

We can use this result to extract a simple formula for the size of $\mathcal{L}_{n, k; r}(132)$:

\begin{cor}
We have that $|\mathcal{L}_{n, k}(132)| = C_n$ for all $n, k \geq 1$ and that $|\mathcal{L}_{n, k; r}(132)| = C_{n + 1} + (k - r - 1)C_n$ for all $n \geq 1$ and  $k - 1 \geq r > 0$.
\end{cor}
\begin{proof}
Both halves of this result follow from Corollary \ref{132prop}: in the case of $\mathcal{L}_{n, k}$, we are counting Young diagrams contained in $\langle n - 1, n - 2, \ldots, 2, 1\rangle$ while in the case of $\mathcal{L}_{n, k; r}$ for $r > 0$ we are counting shapes contained in $\langle n + k - r - 1, n - 1, n - 2, \ldots, 1\rangle$.  As we noted before, the former diagrams are naturally in bijection with Dyck paths and so are a Catalan object; for the latter, consider separately two cases.  First, if the shape $\lambda$ has first row of length at most $n$ then it is one of the $C_{n + 1}$ shapes contained in $\langle n, n - 1, \ldots, 1\rangle$.  Second, if $\lambda$ has first row of length longer than $n$ then we may independently choose any length between $n + 1$ and $n + k - r - 1$ for the first row and any shape contained in $\langle n - 1, n - 2, \ldots, 1\rangle$ for the lower rows and so we have $(k - r - 1)C_n$ additional shapes.
\end{proof}

\subsection{The patterns $312$ and $231$}\label{231sec}

If a shape $\lambda / \mu$ contains a square, every tableau of that shape will have as a sub-tableau four entries
\[
\begin{array}{|c|c|}
\hline
a & b \\
\hline
c & d \\
\hline
\end{array}
\]
with $a < b < d$ and $a < c < d$, and the reading word of every such tableau will be of the form $\cdots cd \cdots ab \cdots$.  But then this permutation contains both an instance $dab$ of the pattern $312$ an instance $cda$ of the pattern $231$.  Thus, the number of tableaux of shape $\lambda / \mu$ whose reading words avoid $312$ or $231$ is zero unless $\lambda / \mu$ contains no square, i.e., unless $\lambda / \mu$ is contained in a ribbon.  In this case, choose a tableau $T$ of shape $\lambda / \mu$ with reading word $w$.  Since $\lambda/\mu$ is a ribbon, the reading word of the conjugate tableau $T'$ is exactly the reverse $w^r$ of $w$.  Since $w$ avoids $312$ if and only if $w^r$ avoids $213$, we can apply Proposition \ref{213prop} to deduce the following:
\begin{prop}\label{312prop}
If skew shape $\lambda / \mu$ is contained in a ribbon then the number of tableaux of shape $\lambda / \mu$ whose reading words avoid the pattern $312$ is equal to the number of partitions whose Young diagram is contained in that of $\mu$.  Otherwise, the number of such tableaux is $0$.
\end{prop}
Either using the same argument that gave us Corollary \ref{132prop} and applying it to Proposition \ref{312prop} or using the same argument that gave us Proposition \ref{312prop} but using Corollary \ref{132prop} in place of Proposition \ref{213prop} gives us the following result:
\begin{cor}\label{231prop}
If skew shape $\lambda / \mu$ is contained in a ribbon then the number of tableaux of shape $\lambda / \mu$ whose reading words avoid the pattern $231$ is equal to the number of partitions whose Young diagram is contained in that of the partition $\langle \lambda_1 - \lambda_k, \lambda_1 - \lambda_{k - 1}, \ldots, \lambda_1 - \lambda_2\rangle$.  Otherwise, the number of such tableaux is $0$.
\end{cor}
 
In the special cases of $\mathcal{L}_{n, k; r}$ it follows that for $k \geq 3$ and $n \geq 2$ we have $\mathcal{L}_{n, k; r}(231) = \mathcal{L}_{n, k; r}(312) = \emptyset$ while for $1 \leq k \leq 2$ we have that $|\mathcal{L}_{n, k; r}(231)|$ and $|\mathcal{L}_{n, k; r}(312)|$ are Catalan numbers \cite{geehoon, catadd}.

\subsection{The patterns $123$ and $321$}\label{123sec}

For the patterns $123$ and $321$, our results are not as nice as those in the preceding sections.  We show that the enumeration of tableaux of given shapes whose reading words avoid these patterns is equivalent to the enumeration of permutations with certain prescribed ascents and descents avoiding these patterns.  This latter problem can easily be solved in any particular case by recursive methods, though the author knows of no closed formula for the resulting values.  As we mentioned in the introduction to Section \ref{skewsec}, the problem of enumerating pattern-avoiding permutations with certain fixed ascents and descents is the special case of our problem in which we consider tableaux contained in a ribbon.

Note that if a shape $\lambda / \mu$ has any rows of length three or more, the reading word of any tableau of shape $\lambda/\mu$ contains a three-term increasing subsequence.  Consequently, in order to study $123$-avoidance it suffices to consider tableaux with all rows of length one or two.  The following result allows us to reduce even further the set of skew shapes we need to consider in order to completely deal with the case of $123$-avoidance.

\begin{prop} If $\lambda/\mu$ is a skew shape, all of whose rows have length one or two, and $\lambda_i - \mu_i = 2$ then the number of tableaux of shape $\lambda/\mu$ whose reading words avoid $123$ is the same as the number of tableaux of the following shapes whose reading words avoid $123$:
\begin{itemize}
\item $\langle \lambda_1 + 1, \ldots, \lambda_i + 1, \lambda_{i+1}, \lambda_{i + 2}, \ldots\rangle / \langle \mu_1 + 1, \ldots, \mu_i + 1, \mu_{i + 1}, \mu_{i+2}, \ldots\rangle$
\item $\langle \lambda_1 + 1, \ldots, \lambda_{i-1} + 1, \lambda_i, \lambda_{i + 1}, \ldots\rangle / \langle \mu_1 + 1, \ldots, \mu_{i-1} + 1, \mu_i, \mu_{i+1}, \ldots\rangle$
\end{itemize}
\end{prop}
In other words, the Proposition states that we can locally slide rows of length two without changing the number of tableaux with $123$-avoiding reading words that result.  Figure \ref{fig:123slide} illustrates these moves.

\begin{figure}
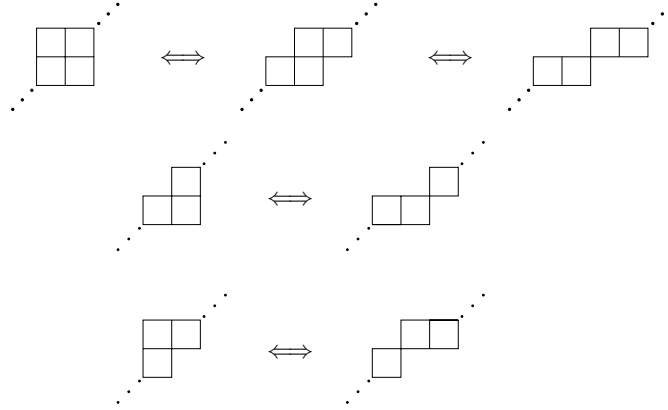

\hspace{3cm}\scalebox{.8}{\input{123slide1.tex}}
\begin{center}
\scalebox{.8}{\input{123slide2.tex}}
\end{center}
\caption{Moves that do not change the set of $123$-avoiding reading words of tableaux.}
\label{fig:123slide}
\end{figure}
\begin{proof}
Given $\lambda$ and $\mu$, let $\alpha = \langle \lambda_1 + 1, \lambda_2 + 1, \ldots, \lambda_i + 1, \lambda_{i+1}, \lambda_{i + 2}, \ldots\rangle$ and $\beta = \langle \mu_1 + 1, \mu_2 + 1, \ldots, \mu_i + 1, \mu_{i + 1}, \mu_{i+2}, \ldots\rangle$ and let $n = |\lambda/\mu| = |\alpha/\beta|$.  There is a natural map between fillings of the shape $\lambda / \mu$ and fillings of the shape $\alpha/\beta$: fill each row of $\alpha/\beta$ with the same numbers in the same order as the corresponding row of $\lambda/\mu$.  We must show that when restricted to standard tableaux whose reading words avoid $123$, this is a bijection.

Given a tableau $T$ of shape $\lambda/\mu$, shifting each of the first $i$ rows one square to the right gives a filling $U$ of $\alpha / \beta$ with $[n]$.  Under this operation, every box is adjacent to the same boxes except for boxes in rows $i$ and $i + 1$.  Thus, the filling $U$ is strictly increasing along rows, and along columns except possibly between rows $i$ and $i + 1$.  By transforming $T$ into $U$, we've moved boxes in row $i$ to the right, so each entry in row $i$ is above a larger entry in $U$ than it was in the tableau $T$, or is above an entry in $T$ but above an empty space in $U$.  It follows that $U$ is also increasing down columns and so is a standard Young tableau.  Thus, the map $T \mapsto U$ sends standard Young tableaux to standard Young tableaux and is injective.  We must show that it is surjective.

Suppose we have a tableau $U$ of shape $\alpha / \beta$ whose reading word avoids $123$.  We show that this tableau is the image under the above-described map of some tableau of shape $\lambda/\mu$, i.e., that shifting the first $i$ rows of $U$ one unit to the left results in a standard Young tableau.  Let $T$ be the filling of $\lambda/\mu$ that results from this shift.  Again, the only possible obstruction is that $T$ might fail to increase along columns between rows $i$ and $i + 1$.  In particular, there are two ways in which this could happen.  Either the rightmost of the two entries in the $i$th row of $U$ (the entry $U(i, \mu_i + 2) = U(i, \lambda_i)$) is larger than the entry $U(i, \mu_i + 1)$ immediately below it and to its left or the leftmost of the two entries in the $i$th row of $U$ (the entry $U(i, \mu_i + 1) = U(i, \lambda_i - 1)$) is larger than the entry of $U(i + 1, \mu_i)$ immediately below it and to its left.  Suppose for sake of contradiction that one of these two situations occurs.

In the former case, since $\alpha/\beta$ is the image of a skew shape $\lambda/\mu$ under a shift to the first $i$ rows, row $i + 1$ must have length two.  Thus, the reading word of $U$ has the form $\cdots U(i + 1, \mu_i)U(i + 1, \mu_i + 1)\cdots U(i, \mu_i + 2) \cdots$ and the subsequence $U(i + 1, \mu_i)U(i + 1, \mu_i + 1)U(i, \mu_i + 2)$ is a $123$ pattern, a contradiction with the choice of $U$.  In the latter case, the reading word of $U$ has the form $\cdots U(i + 1, \mu_i) \cdots U(i, \mu_i + 1) U(i, \mu_i + 2) \cdots$, and the subsequence $U(i + 1, \mu_i) U(i, \mu_i + 1) U(i, \mu_i + 2)$ is a $123$ pattern, again a contradiction.  Thus neither potential obstruction can occur, so $U$ is the image of a standard Young tableau $T$ and the map is surjective, as desired.

The other case is extremely similar, and we omit it here.
\end{proof}

It is an immediate consequence of this proposition that the problem of counting tableaux of a given skew shape whose reading words avoid $123$ reduces to the case of shapes contained in ribbons.  This is equivalent to counting $123$-avoiding permutations with certain prescribed ascents and descents, and this enumeration can be carried out by straightforward recursive methods.

A similar analysis can be applied to the pattern $321$.  Here we shift the first $i$ or $i - 1$ columns down (instead of shifting the first $i$ or $i - 1$ rows to the right).  This results in a modest increase in the care needed to make the argument work.  In particular, shifting columns does not preserve reading words.  Luckily, in the $321$-avoiding case all of our columns have length two, so the reading words change in a relatively controlled way; for example, a $2 \times 2$ square with reading word $3412$ will, after shifting, have reading word $3142$.  The interested reader is invited to work out the details for herself.

\section{Acknowledgments}
We would like to thank Suho Oh, Craig Desjardins, Alejandro Morales and especially Alex Postnikov for helpful discussions leading to the formulation of these results.  We would also like to thank two anonymous referees for a number of helpful suggestions.

\bibliographystyle{plain}
\bibliography{altpatperms}{}

\begin{thebibliography}{10}

\bibitem{Romik}
Yuliy Baryshnikov and Dan Romik.
\newblock Enumeration formulas for {Y}oung tableaux in a diagonal strip.
\newblock {\em Israel J. of Mathematics}, 178:157--186, 2010.

\bibitem{1342}
Mikl\'os B\'ona.
\newblock Exact enumeration of 1342-avoiding permutations; a close link with
  labeled trees and planar maps.
\newblock {\em J. Combinatorial Theory, Series A}, 80:257--272, 1997.

\bibitem{kernelgentrees}
Mireille Bousquet-M\'elou.
\newblock Four classes of pattern-avoiding permutations under one roof:
  generating trees with two labels.
\newblock {\em Electronic J. Combinatorics}, 9:R19, 2003.

\bibitem{Foulkes}
Herbert~O. Foulkes.
\newblock Tangent and secant numbers and representations of symmetric groups.
\newblock {\em Discrete Mathematics}, 15:311--324, 1976.

\bibitem{1234}
Ira~M. Gessel.
\newblock Symmetric functions and {P}-recursiveness.
\newblock {\em J. Combinatorial Theory, Series A}, 53:257--285, 1990.

\bibitem{geehoon}
Geehoon Hong.
\newblock Catalan numbers in pattern-avoiding permutations.
\newblock {\em MIT Undergraduate J. Mathematics}, 10:53--68, 2008.

\bibitem{self}
Joel~Brewster Lewis.
\newblock Alternating, pattern-avoiding permutations.
\newblock {\em Electronic J. Combinatorics}, 16:N7, 2009.

\bibitem{JCTAversion}
Joel~Brewster Lewis.
\newblock Pattern avoidance for alternating permutations and {Y}oung tableaux.
\newblock {\em J. Combinatorial Theory, Series A}, 118(4):1436--1450, May 2011.

\bibitem{mans}
Toufik Mansour.
\newblock Restricted $132$-alternating permutations and {C}hebyshev
  polynomials.
\newblock {\em Annals of Combinatorics}, 7:201--227, 2003.

\bibitem{der}
Toufik Mansour and Aaron Robertson.
\newblock Refined restricted permutations avoiding subsets of patterns of
  length three.
\newblock {\em Annals of Combinatorics}, 6:407--418, 2002.

\bibitem{ouch}
Erik Ouchterlony.
\newblock Pattern avoiding doubly alternating permutations.
\newblock \texttt{arXiv:0908.0255v1}, 2006.

\bibitem{sagan}
Bruce Sagan.
\newblock {\em The Symmetric Group}.
\newblock Springer-Verlag, 2001.

\bibitem{simion}
Rodica Simion and Frank~W. Schmidt.
\newblock Restricted permutations.
\newblock {\em European J. Combinatorics}, 6:383--406, 1985.

\bibitem{EC1}
Richard~P. Stanley.
\newblock {\em Enumerative Combinatorics, Volume 1}.
\newblock Cambridge University Press, 1997.

\bibitem{EC2}
Richard~P. Stanley.
\newblock {\em Enumerative Combinatorics, Volume 2}.
\newblock Cambridge University Press, 2001.

\bibitem{StanAltSym}
Richard~P. Stanley.
\newblock Alternating permutations and symmetric functions.
\newblock {\em J. Combinatorial Theory, Series A}, 114:436--460, 2007.

\bibitem{catadd}
Richard~P. Stanley.
\newblock Catalan addendum to \emph{Enumerative Combinatorics}.
\newblock Available online at \texttt{www-math.mit.edu/$\sim$\!
  rstan/ec/catadd.pdf}, 2009.

\bibitem{west1995}
Julian West.
\newblock Generating trees and the {C}atalan and {S}chr{\"{o}}der numbers.
\newblock {\em Discrete Mathematics}, 146:247--262, 1995.

\end{thebibliography}

\end{document}